\documentclass[final,1p,sort&compress]{elsarticle}

\usepackage{lipsum}
\usepackage[utf8]{inputenc}
\usepackage{amsfonts}
\usepackage{algorithm}
\usepackage{algorithmic}
\usepackage[english]{babel}
\usepackage{lineno}
\usepackage{graphicx,epstopdf,epsfig}
\usepackage{amsfonts,epsfig,fancyhdr,graphics, hyperref,amsmath,amssymb}
\usepackage[toc,page]{appendix}
\usepackage{todonotes}
\usepackage{booktabs}
\usepackage{tabularx}
\ifpdf
  \DeclareGraphicsExtensions{.eps,.pdf,.png,.jpg}
\else
  \DeclareGraphicsExtensions{.eps}
\fi

\newcommand{\E}{{\rm e}}
\DeclareMathOperator*{\diag}{diag}

\usepackage{amsthm}
\newtheorem{Theorem}{Theorem}
\newtheorem{Lemma}{Lemma}
\newtheorem{Proposition}{Proposition}

\theoremstyle{definition}
\newtheorem{Definition}{Definition}
\newtheorem{Remark}{Remark}

\usepackage{amsopn}
\usepackage{tocbasic}
\renewcommand*{\tableofcontents}{\listoftoc[\contentsname]{toc}}
\DeclareTOCStyleEntries[raggedentrytext]{tocline}{section,subsection,subsubsection,paragraph,subparagraph}

\journal{Applied Numerical Mathematics}
\begin{document}

\begin{frontmatter}


\title{Symbol based Convergence Analysis in Block Multigrid Methods with Applications for Stokes Problems\tnoteref{financing}}
\tnotetext[financing]{The work of the second, third, and fourth authors is partly supported by “Gruppo Nazionale per il Calcolo Scientifico" (GNCS-INdAM). The work of Isabella Furci  was carried out within the framework of the project “A multiscale integrated approach to the study of the nervous system in health and disease (MNESYS)” and has been supported by European Union - NextGenerationEU.}

\author[mbolt]{Matthias Bolten}
\ead{bolten@uni-wuppertal.de}
\address[mbolt]{School of Mathematics and Natural Sciences, University of Wuppertal, Wuppertal, Germany}

\author[mdon]{Marco Donatelli}
\ead{marco.donatelli@uninsubria.it}
\address[mdon]{Department of Science and High Technology. University of Insubria, Como, Italy.}

\author[mbolt]{Paola Ferrari}
\ead{ferrari@uni-wuppertal.de}

\author[ifurc]{Isabella Furci\corref{cor1}}
\ead{isabella.furci@dima.unige.it}
\address[ifurc]{Department of Mathematics. University of Genoa, Genoa, Italy.}

\cortext[cor1]{Corresponding author}

\begin{abstract}
 The main focus of this paper is the study of efficient multigrid methods for large linear systems with a particular saddle-point structure. {Indeed, when the system matrix is symmetric, but indefinite, the variational convergence theory that is usually used to prove multigrid convergence cannot be directly applied.}
 
 {However, different algebraic approaches analyze properly preconditioned saddle-point problems, proving convergence of the
Two-Grid method. In particular, this is  efficient when the blocks of the coefficient matrix possess a Toeplitz or circulant structure.}   Indeed, it is possible to derive sufficient conditions for convergence and provide optimal parameters for the preconditioning of the
saddle-point problem in terms of the associated  symbols.
In this paper, we propose a symbol based convergence analysis for problems that have a hidden block Toeplitz structure. Then, they can be investigated focusing on the properties of the associated generating function $\mathbf{f}$, which consequently is a matrix-valued function with dimension depending on the block size of the problem.  As numerical tests we focus on the matrix sequence stemming from  the finite element approximation of the Stokes problem. We show the efficiency of the methods studying the hidden  9-by-9 block multilevel structure of the obtained matrix sequence. Moreover we propose an efficient algebraic multigrid method with convergence rate independent of the matrix size. Finally, we present several numerical tests comparing the results with state-of-the-art strategies.

\end{abstract}

\begin{keyword}
Multigrid methods, block matrix sequences, spectral symbol, Stokes equation
\MSC[2010] 65N55, 15B05, 35P20, 34M40

\end{keyword}

\end{frontmatter}


\section{Introduction}\label{sec:intr}
Many works have focused on efficient iterative methods for the solution of large linear systems with coefficient matrix in saddle-point form \cite{MR3564863, MR3439215, benzi2005numerical}. Indeed, many preconditioned Krylov subspace methods have been developed based on the study of the spectral properties of the coefficient matrix \cite{MR3564863,MR2220678,benzi2006eigenvalues,locspsb1}. When we deal with multilevel and multilevel block systems the performance of some known preconditioners can deteriorate \cite{NSV}. 
Multigrid methods (MGM) are in general a good choice in terms of computational cost and optimal convergence rate. Indeed, 
optimal convergence rates are obtained by constructing and employing a proper sequence of linear systems of decreasing dimensions. Moreover, when the coefficient matrix has a particular multilevel (block) Toeplitz structure, the convergence analysis for such methods can be obtained in compact form exploiting the concept of symbols \cite{FS1,FS2}.
 
The knowledge and the use of the symbol has been successfully employed for designing very efficient preconditioning and multigrid techniques in various settings and several applications \cite{MR3689933,DFFSS,MR4188673,sesana2013spectral}. However, {saddle-point problems are intrinsically indefinite. This limits the possible convergence results that can be exploited, even in the case where we can benefit from the Toeplitz/Toeplitz-like structure of the coefficient matrix. }

In many works \cite{MR3564863, MR2220678, MR3689933} a common choice for problems with saddle-point structure is that of constructing preconditioning for two-by-two block systems involving the Schur complement. In many cases stemming from the approximation of Partial Differential Equations (PDEs), even though the whole problem is indefinite, the latter procedure is based on the solution of a linear system with a positive definite coefficient matrix.

{In this paper we consider the approach suggested  by \cite{MR3439215} which  provides a point-wise smoother and a
coarse grid correction to the linear system after preconditioning from the left and the right using lower and upper triangular matrices,
respectively.}

{This strategy is particularly effective in the case where the blocks of the matrix possess a scalar unilevel Toeplitz and circulant structure. Indeed, in \cite{saddle_DBFF} sufficient conditions for the Two-Grid Method (TGM) convergence and “level independency” for the W-cycle have been obtained in terms of the generating functions of
the  blocks. Studying the latter functions, also the optimal parameters for the preconditioning of the
saddle-point problem and for the point smoother were found. Clearly, in the  scalar setting, the involved generating functions were {univariate or multivariate} scalar trigonometric polynomials. Consequently, the exploited MGM theory is that for structured {unilevel or multilevel}  scalar matrices \cite{AD,ADS}. 
When  discretizing a PDE  with higher order approximation techniques}, the coefficient matrix - or its blocks - possess a (multilevel)  block structure. 

In this paper we focus on saddle-point problems where the blocks of the coefficient matrix have a (hidden) multilevel block structure. Consequently, we combine the method given in \cite{MR3439215} with the recent results on TGM convergence analysis for block structured matrices \cite{block_multigrid2,DFFSS}. Moreover, we show that the structure can be preserved on the coarse level, allowing for a recursive
application of the method, after proper symmetrization. The global algorithm is then a 3-step procedure whose main body is formed by a symbol based block multigrid method. 

The outline of the paper is as follows.  In Subsections {\ref{ssec:intr}-- \ref{ssec:saddle_uni_scalar} we provide the notation and the necessary results on multigrid methods for both saddle-point and symmetric definite positive (SPD) matrices. Regarding SPD matrices} we focus especially on the case where the multigrid is applied to Toepliz matrices with {both scalar or matrix-valued generating functions. At the end of the Section \ref{sec:intr} we report the main results of the TGM when applied to saddle-point problems.}
In Section \ref{sec:saddle_multi_block}
 we provide TGM convergence results for saddle-point systems with blocks that are multilevel block structured matrices.
In Section \ref{sec:stokes} we introduce the Stokes problem in 2D and we focus on a Finite Element approximation in Subsection \ref{ssec:Q1_iso_Q1_matthias}. The discrete problem involves solving a saddle-point linear system with block multilevel Toeplitz blocks. We theoretically show that such a problem is within the setting and fulfils the assumptions of the TGM convergence results described in Section \ref{sec:saddle_multi_block}.  In Section \ref{sec:procedure}, we introduce an iterative 3-step procedure that can be exploited for solving recursively linear systems with structure {described as in Subsection \ref{ssec:Q1_iso_Q1_matthias}. Precisely, we focus on the structure and on the spectral features of the involved matrix sequences in order to prove the efficiency of the multigrid strategy with the proposed projection operators given in Subsection \ref{ssec:smoth_proj_mgm}. Finally, we conclude with Section \ref{sec:numerical} presenting selected numerical tests that confirm the efficiency of the proposed multigrid method in comparison with a state-of-the-art multigrid strategy for saddle-point structured problems.

\subsection{Notation and preliminary results}\label{ssec:intr}
This subsection is devoted to fix the notation and present the key concepts on multigrid methods and Toeplitz matrices.

\label{sec:introduction:tep_circ}
We report the main definitions and properties of Toeplitz and circulant matrices generated by a function $f$. The domain of $f$ can be either one-dimensional $[-\pi,\pi]$ or $d$-dimensional $[-\pi,\pi]^d$ and this gives either unilevel or $d$-level structured matrices, respectively. The co-domain of $f$ can be either the complex field or the linear space of $d \times d$ complex matrices, generating scalar or block structured matrices, respectively. We will denote the generating function by bold $\mathbf{f}$ when it is a matrix-valued function. When we want to highlight that the generating function is multivariate, we specify its argument $\mathbf{f}(\boldsymbol{\theta})$ in bold, where $\boldsymbol{\theta}=(\theta_1,\theta_2, \dots, \theta_d)\in [-\pi,\pi]^d$. The structure and the size of the Toeplitz matrix is clearly related to its generating function.

In the whole paper we adopt the following notation on norms and relations between Hermitian Positive Definite (HPD) matrices. If $X\in\mathbb{C}^{N\times N}$ is a HPD matrix, then
$\|\cdot\|_{X}=\|X^{1/2}\cdot\|_{2}$ denotes the Euclidean norm weighted by $X$ on $\mathbb{C}^{N}$. 
If $X$ and $Y$ are Hermitian matrices, then
the notation $X\leq Y$ means that $Y-X$ is a nonnegative definite matrix. Finally, given a matrix-valued function $\mathbf{f}$ defined on the cube $[-\pi,\pi]^m$ we denote by $ \|\mathbf{f}\|_{\infty}={\rm ess\,sup}_{\boldsymbol{\theta}\in [-\pi,\pi]^m}\|\mathbf{f}(\boldsymbol{\theta})\|_2$.

 In what follows, we report the formal definition of Toeplitz and circulant matrices.


\begin{Definition}\label{def:toeplitz_block_generating}
Given a Lebesgue integrable function $\mathbf{f}:[-\pi,\pi]^{d}\rightarrow \mathbb{C}^{s\times s}$, its Fourier coefficients are given by
\begin{linenomath*} 
\begin{equation*}
  \hat{\mathbf{f}}_{\mathbf{k}} =
  \frac1{(2\pi)^d}
  \int_{[-\pi,\pi]^d}\mathbf{f}(\boldsymbol{\theta})\E^{-\iota\left\langle {\mathbf{k}},\boldsymbol{\theta}\right\rangle}\mathrm{d}\boldsymbol{\theta}\in\mathbb{C}^{s\times s},
  \qquad \mathbf{k}=(k_1,\ldots,k_d)\in\mathbb{Z}^d,\label{fhat}
\end{equation*}\end{linenomath*}
where $\boldsymbol{\theta}=(\theta_1,\ldots,\theta_d)$, $\left\langle \mathbf{k},\boldsymbol{\theta}\right\rangle=\sum_{i=1}^dk_i\theta_i$, and the integrals of matrices are computed elementwise.

The $\mathbf{n}$th Toeplitz matrix, where $\mathbf{n}=(n_1,\ldots,n_d)$, associated with $\mathbf{f}$ is the matrix of dimension $s \cdot n_1\cdot\ldots\cdot n_d$ given by
\begin{linenomath*} 
\begin{equation*}
T_\mathbf{n}(\mathbf{f})=
\sum_{\mathbf{1}-\mathbf{n}\le \mathbf{k}\le \mathbf{n}-\mathbf{1}} J_{n_1}^{k_1} \otimes \cdots\otimes J_{n_d}^{k_d}\otimes \hat{\mathbf{f}}_{\mathbf{k}},
\end{equation*}\end{linenomath*}
where $\mathbf{1}$ is the vector of all ones, $ J^{k_{\xi}}_{n_\xi}$  is the $n_{\xi} \times n_{\xi}$ matrix whose $(i,h)$th entry equals 1 if $(i-h)=k_{\xi}$ and $0$ otherwise, and where $\mathbf{s}\le \mathbf{t}$ means that $s_j\le t_j$ for any $j=1,\ldots,d$.

The $\mathbf{n}$th circulant matrix associated with $\mathbf{f}$ is the matrix of dimension $s \cdot n_1\cdot\ldots\cdot n_d$ given by
\begin{linenomath*} 
\begin{equation*}
\mathcal{C}_\mathbf{n}(\mathbf{f})=
\sum_{\mathbf{1}-\mathbf{n}\le \mathbf{k}\le \mathbf{n}-\mathbf{1}} Z_{n_1}^{k_1} \otimes \cdots\otimes Z_{n_d}^{k_d}\otimes \hat{\mathbf{f}}_{\mathbf{k}},
\end{equation*}\end{linenomath*}
where $Z^{k_{\xi}}_{n_\xi}$  is the $n_{\xi} \times n_{\xi}$ matrix whose $(i,h)$th entry equals 1 if $(i-h) \mod n_{\xi}=k_{\xi}$ and $0$ otherwise.
\end{Definition}

\begin{Theorem}[\cite{davis}]\label{circ-s}
Let $\textbf{f}\in L^1([-\pi,\pi]^d,s)$ be a matrix-valued function with $d\ge 1,s\ge 2$. Then, the following (block-Schur) decomposition of
$\mathcal{C}_{\bf n}(\textbf{f})$ is valid:
\begin{equation}
\label{schur-s}
\mathcal{C}_{\bf n}(\textbf{f})=(F_{\bf n}\otimes I_s) D_{\bf n}(\textbf{f}) (F_{\bf n}\otimes I_s)^H,
\end{equation}
where
\begin{equation}\label{eig-circ-s}
  D_{\bf n}(\textbf{f}) =\underset{{ {\bf 0}\le {\bf k}\le {\bf n}-{\bf 1}}}{{\rm diag}}\left(S_{\bf n}(\textbf{f})\left(\theta_{\bf k}^{({\bf n})}\right)\right),
\end{equation}
$F_{\bf n}$ is the d-level Fourier matrix, $F_{\bf n} = F_{n_1} \otimes\dots \otimes F_{n_d}$, $I_s$ is the identity matrix of size $s$
 and
\begin{equation}
\label{eq:grid}
\theta_{\bf k}^{({\bf n})}=2\pi\frac{ {\bf k}}{{\bf n}}, \quad  { {\bf k}={\bf 0}, \dots, {\bf n}-{\bf 1}}.
\end{equation}
Moreover, $S_{\bf n}(\textbf{f})(\cdot)$ is the $\bf n$-th Fourier sum of $\textbf{f}$ given by
\begin{equation}\label{fourier-sum-s}
S_{\bf n}(f)(\boldsymbol{\theta}) = \sum_{j_1=1-n_1}^{n_1-1} \cdots  \sum_{j_d=1-n_d}^{n_d-1} \hat f_{\bf j}
{\rm e}^{{\iota}\left\langle {\bf j},\boldsymbol{\theta}\right\rangle}, \ \ \ \  \left\langle {\bf j},\theta\right\rangle=\sum_{t=1}^d  j_t \theta_t.
\end{equation}
\end{Theorem}

\begin{Remark}\label{rmk:fourier-vs-funzione}
The latter results states that the eigenvalues of $\mathcal{C}_{\bf n}(\textbf{f})$ are given by the evaluations of $\lambda_t(S_{\bf n}(\textbf{f})(\cdot))$, $t=1,\ldots,s$, at the grid points \eqref{eq:grid}.

Moreover, if $\textbf{f}$ is a trigonometric polynomial of fixed degree (with respect to $\bf n$), then it is worth noticing that
 $S_{\bf n}(\textbf{f})(\cdot) = \textbf{f}(\cdot)$ for $\bf n$ large enough. Therefore, in such a setting, the eigenvalues of $\mathcal{C}_{\bf n}(\textbf{f})$ are the evaluations of $\lambda_t(\textbf{f}(\cdot))$, $t=1,\ldots,s$, at the very same grid points. The latter implies that if $\textbf{f}\ge 0$ trigonometric polynomial and $\lambda_{\min}(\textbf{f})$ does not vanish at any of the grid points (\ref{eq:grid}), then $\mathcal{C}_{\bf n}(\textbf{f})$ is HPD.
\end{Remark}


\begin{Remark}\label{rmk:rectangular_Cn}
	In Definition \ref{def:toeplitz_block_generating}, we considered square matrix-valued functions. In the general case where $\mathbf{f}:[-\pi,\pi]^{d}\rightarrow \mathbb{C}^{s_1\times s_2}$, when we write $\mathcal{C}_\mathbf{n}(\mathbf{f})$, we mean that 
	\begin{equation}\label{eq:rectangular_Cn}
		\mathcal{C}_\mathbf{n}(\mathbf{f})=\boldsymbol{\Pi}^T_{s_1n^d\times s_1n^d}\begin{bmatrix}
\mathcal{C}_{\mathbf{n}}({f}^{(1,1)}) &  \mathcal{C}_{\mathbf{n}}({f}^{(1,2)}) &\dots & \mathcal{C}_{\mathbf{n}}({f}^{(1,s_2)})\\
\mathcal{C}_{\mathbf{n}}({f}^{(2,1)}) & \mathcal{C}_{\mathbf{n}}({f}^{(2,2)}) & \dots & \mathcal{C}_{\mathbf{n}}({f}^{(2,s_2)}) \\
\vdots& &\ddots& \vdots\\
\mathcal{C}_{\mathbf{n}}({f}^{(s_1,1)}) &  \mathcal{C}_{\mathbf{n}}({f}^{(s_1,2)}) &\dots & \mathcal{C}_{\mathbf{n}}({f}^{(s_1,s_2)})
\end{bmatrix} \boldsymbol{\Pi}_{s_2n^d\times s_2n^d},
	\end{equation}
where ${f}^{(j,k)}$ is the $(j,k)$th component of the function $\mathbf{f}$ and, for $s\in\{s_1,s_2\}$, the permutation matrices $\boldsymbol{\Pi}_{sn^d\times sn^d}$ are defined by
\begin{equation}
	\boldsymbol{\Pi}_{sn^d\times sn^d} = [\Pi_1|\Pi_2|\dots|\Pi_{n^d}],
\end{equation}
with
\[
		[\Pi_k]_{i,j} = \begin{cases}
							1 & \mbox{if } i = (j-1)n^d+k,\\
							0 & \mbox{otherwise},
						\end{cases}, \qquad
		i = 1, \dots, sn^d, \quad j = 1, \dots, s.
\]

This approach is consistent with Definition \ref{def:toeplitz_block_generating}, however, we avoided to include non-square matrices in the definition of circulants because in the next sections we will exploit the circulant algebra structure, which is natural in the square case.
\end{Remark}

\begin{Remark}\label{rmk:rectang_Cn_prod}
	In the non-square case it is still possible to exploit symbols to perform products between circulant matrices as follows.
	Suppose $\mathbf{f}:[-\pi,\pi]^{d}\rightarrow \mathbb{C}^{s_1\times s_2}$, $\mathbf{g}:[-\pi,\pi]^{d}\rightarrow \mathbb{C}^{s_2\times s_3}$ and denote the components of the product $\mathbf{f}\mathbf{g}$ with ${h}^{(i,j)}$, then
	\begin{align*}
		\mathcal{C}_\mathbf{n}(\mathbf{f})\mathcal{C}_\mathbf{n}(\mathbf{g})  =&
		\boldsymbol{\Pi}^T_{s_1n^d\times s_1n^d}\begin{bmatrix} \mathcal{C}_{\mathbf{n}}({f}^{(i,j)}) \end{bmatrix}_{\substack{i=1,\dots,s_1 \\ j = 1,\dots,s_2}}\boldsymbol{\Pi}_{s_2n^d\times s_2n^d} \cdot\\
		&\cdot\boldsymbol{\Pi}^T_{s_2n^d\times s_2n^d}\begin{bmatrix} \mathcal{C}_{\mathbf{n}}({g}^{(i,j)}) \end{bmatrix}_{\substack{i=1,\dots,s_1 \\ j = 1,\dots,s_2}}\boldsymbol{\Pi}_{s_3n^d\times s_3n^d}\\
		=&\boldsymbol{\Pi}^T_{s_1n^d\times s_1n^d}\begin{bmatrix} \mathcal{C}_{\mathbf{n}}({f}^{(i,j)}) \end{bmatrix}_{\substack{i=1,\dots,s_1 \\ j = 1,\dots,s_2}}\begin{bmatrix} \mathcal{C}_{\mathbf{n}}({g}^{(i,j)}) \end{bmatrix}_{\substack{i=1,\dots,s_1 \\ j = 1,\dots,s_2}}\boldsymbol{\Pi}_{s_3n^d\times s_3n^d}\\
		=&\boldsymbol{\Pi}^T_{s_1n^d\times s_1n^d}\begin{bmatrix} \mathcal{C}_{\mathbf{n}}({h}^{(i,j)}) \end{bmatrix}_{\substack{i=1,\dots,s_1 \\ j = 1,\dots,s_2}}\boldsymbol{\Pi}_{s_3n^d\times s_3n^d},\\
		=&\mathcal{C}_\mathbf{n}(\mathbf{f}\mathbf{g}).
	\end{align*}
\end{Remark}

The generating function plays an important role for investigating some analytic properties of the related Toeplitz matrix. Under particular hypotheses, $f$ can provide important spectral information on $T_n(f)$. In this case we say that $f$ is the \textit{spectral symbol} of the matrix sequence $\{T_n(f)\}_n$. In the following, we report an important localization result on the eigenvalues of $T_{\bf n}(\mathbf{f})$ generated by a Hermitian $d\times d$ matrix-valued function.  

\begin{Theorem}[\cite{serra1998, serra99}]\label{thm:loc-extr-s}
Let $\mathbf{f}\in L^1([-\pi,\pi]^m)$ be a Hermitian $d\times d$ matrix-valued function with $m\ge 1,d\ge 2$. Let $m_1$ be the essential infimum of 
the minimal eigenvalue of $\mathbf{f}$, $M_1$ be the essential supremum of the minimal eigenvalue of $\mathbf{f}$, $m_d$ be the essential infimum of 
the maximal eigenvalue of $\mathbf{f}$, and $M_d$ be the essential supremum of the maximal eigenvalue of $\mathbf{f}$. 
\begin{enumerate}
\item If $m_1<M_1$ then all the eigenvalues of $T_{\bf n}(\mathbf{f})$ belong to the open set $(m_1,M_d]$ for every ${\bf n}\in\mathbb N^m$. If $m_d<M_d$ then all the eigenvalues of $T_{\bf n}(\mathbf{f})$ belong to the open set $[m_1,M_d)$ for every ${\bf n}\in\mathbb N^m$.
\item If $m_1=0$ and $\boldsymbol{\theta}_0$ is the unique zero of $\lambda_{\min}(\mathbf{f})$ such that there exist positive constants $c,C,t$ for which
\[ 
c\|\boldsymbol{\theta} -\boldsymbol{\theta}_0\|^t \le \lambda_{\min}(f(\boldsymbol{\theta})) \le C \|\boldsymbol{\theta} -\boldsymbol{\theta}_0\|^t,
\]
then the minimal eigenvalue of $T_{\bf n}(\mathbf{f})$ goes to zero as $(n_1n_2\dots n_m)^{-t/m}$.
\end{enumerate}
\end{Theorem} 
\subsection{Multigrid methods}\label{ssec:multigrid}
In this subsection we briefly report the relevant results concerning the convergence theory of
 algebraic multigrid  methods \cite{RStub}  especially when the coefficient matrix is a  block-Toeplitz  matrix  generated  by  a matrix-valued trigonometric polynomial \cite{block_multigrid2}. 
 
In the general case, we are interested in solving a linear system $A_{N_0}x_{N_{0}}=b_{N_{0}}$  where $A_{N_0}$ is a Hermitian positive definite matrix. Assume $N_0>N_1>\dots >N_{\ell}>\dots > N_{\ell_{\min}}$ and define two sequences of full-rank rectangular matrices $P_{N_{\ell},N_{\ell+1}}\in \mathbb{C}^{N_{\ell}\times N_{\ell+1}}$ and $R_{N_{\ell+1},N_{\ell}}\in \mathbb{C}^{N_{\ell+1}\times N_{\ell}}$. 
   Let us consider the stationary iterative methods $\mathcal{V}_{N_\ell,\rm{pre}}$, with iteration matrix ${V}_{N_\ell,\rm{pre}}$, and $\mathcal{V}_{N_\ell,\rm{post}}$, with iteration matrix ${V}_{N_\ell,\rm{post}}$.

An iteration of an algebraic Multigrid Method (MGM) is given by Algorithm~\ref{alg:mgm}.

 \begin{small}
\begin{algorithm}
	\caption{MGM$(A_{N_{\ell}},\mathcal{V}_{N_{\ell},\rm{pre}}^{\nu_{\rm{pre}}},\mathcal{V}_{N_{\ell},\rm{post}}^{\nu_{\rm{post}}},R_{N_{\ell+1},N_{\ell}},P_{N_{\ell},N_{\ell+1}},
	b_{N_{\ell}},x_{N_{\ell}}^{(j)},\ell)$}
	\label{alg:mgm}
	\begin{algorithmic}
\begin{small}
		\STATE{ 0. \textbf{if} $\ell= \ell_{\min}$}
	\STATE{1. $\,\quad {\rm \mathbf{solve}} \, A_{N_{\ell}} x_{N_{\ell}}^{(j+1)}=b_{N_{\ell}}$}
	\STATE{ 2. \textbf{else}}
		\STATE{ 3. $ \,\quad		
		\tilde{x}_{N_{\ell}}=\mathcal{V}_{N_{\ell},\rm{pre}}^{\nu_{\rm{pre}}}(A_{N_{\ell}},{b}_{N_{\ell}},x_{N_{\ell}}^{(j)})$}
		\STATE{ 4. $ \,\quad	r_{N_{\ell}}=b_{N_{\ell}}-A_{N_{\ell}}\tilde{x}_{N_{\ell}}$}
		\STATE{ 5. $\,\quad	r_{N_{\ell+1}}=R_{N_{\ell+1},N_{\ell}}r_{N_{\ell}}$}
		\STATE{ 6. $ \,\quad	A_{N_{\ell+1}}=R_{N_{\ell+1},N_{\ell}}A_{N_{\ell}}P_{N_{\ell},N_{\ell+1}}$}
\STATE{ 7. $\quad$ \textbf{for} $i=1:\gamma$}
	\STATE{8. $\,\quad \quad$ MGM$(A_{N_{\ell+1}},\mathcal{V}_{N_{\ell+1},\rm{pre}}^{\nu_{\rm{pre}}},\mathcal{V}_{N_{\ell+1},\rm{post}}^{\nu_{\rm{post}}},	R_{N_{\ell+2},N_{\ell+1}},
	P_{N_{\ell+1},N_{\ell+2}},
	b_{N_{\ell+1}},x_{N_{\ell+1}}^{(j)},\ell+1)$}
		\STATE{ 9. $\,\,\quad $\textbf{end}}			
		\STATE{ 10. $\quad	\hat{x}_{N_{\ell}}=\tilde{x}_{N_{\ell}}+P_{N_{\ell},N_{\ell+1}}y_{N_{\ell+1}}$}
		\STATE{ 11. $\quad	x_{N_{\ell}}^{(j+1)}=\mathcal{V}_{N_{\ell},\rm{post}}^{\nu_{\rm{post}}}(A_{N_{\ell}},{b}_{N_{\ell}},\hat{x}_{N_{\ell}})$}
	
	\end{small}	
	\end{algorithmic}
\end{algorithm}
\end{small}

 The steps in lines $3$ and $11$ consist, respectively, in applying $\nu_{\rm{pre}}$ times a pre-smoother and $\nu_{\rm{post}}$ times a post-smoother of the given iterative methods. 
The steps in lines $4$ to $10$ define the ``coarse grid correction''.
When $\gamma$ and $\ell_{\min}$ are equal to  1 the Algorithm   \ref{alg:mgm} represents one iteration of the Two-Grid method (TGM). In this case a linear system with coefficient matrix $A_{N_{1}}$  has to be solved. Consequently, if the coefficient matrix has a large size, $\ell_{\min}$ should be taken such that $N_{\ell_{\min}}$ becomes
small enough for solving cheaply step 1. 
In the latter situation, the choice $\gamma=1$ defines the V-cycle method and the choice $\gamma=2$ defines the W-cycle method.

The next Theorem provides the conditions which are sufficient for the convergence of the TGM in the particular case where the matrices $A_{N_{\ell+1}}$ at the coarse levels are computed according to the Galerkin approach ($R_{N_{\ell+1},N_{\ell}}=P_{N_{\ell},N_{\ell+1}}^H$). Moreover, the theorem guarantees the linear convergence of the method, which means that the number of iterations needed by the algorithm to reach a given accuracy $\epsilon$ is  bounded from above by a constant independent of the size of the problem. In this case we will say that the method is convergent and optimal.
\begin{Theorem}\label{thm:teoconv}(\cite{RStub})
	Let $A_{N_{\ell}}$ and  $V_{N_{\ell},{\rm post}}$ be defined as in the {\rm TGM} algorithm and assume that no pre-smoothing is applied $(V_{N_{\ell},\rm{pre}}^{\nu_{\rm{pre}}}=I_{N_{\ell}})$. Let us define 
\begin{small}
\begin{multline*}
{\rm TGM}(A_{N_{\ell}},I_{N_{\ell}},V_{N_{\ell},\rm{post}}^{\nu_{\rm{post}}},
P_{N_{\ell},N_{\ell+1}})\\=
V_{N_{\ell},\rm{post}}^{\nu_{\rm{post}}}
\left[I_{N_{\ell}}-P_{N_{\ell},N_{\ell+1}}\left(P_{N_{\ell},N_{\ell+1}}^{H}
A_{N_{\ell}}P_{N_{\ell},N_{\ell+1}}\right)^{-1}P_{N_{\ell},N_{\ell+1}}^{H}A_{N_{\ell}}\right],
\end{multline*}
\end{small}
the iterative matrix of the TGM. Assume that
	\begin{itemize}
	
		\item[(a)] the smoother fulfills the {\rm smoothing property}: exists  $\sigma_{\rm{post}}>0$ such that 
		\begin{small}
		\[\|V_{N_{\ell},\rm{post}}x_{N_{\ell}}\|_{A_{N_{\ell}}}^{2}\leq\|x_{N_{\ell}}\|_{A_{N_{\ell}}}^{2}-\sigma_{\rm{post}}\|x_{N_{\ell}}\|_{A_{N_{\ell}}^2}^{2},\quad \forall x_{N_{\ell}}\in\mathbb{C}^{N_{\ell}}.\]\end{small}
		\item[(b)] The grid transfer operator $P_{N_{\ell},N_{\ell+1}})$ fulfills the {\rm approximation property}: exists $ {\kappa(A_{N_\ell},P_{{N_\ell,{N_{\ell+1}}}})}>~0$ such that
		\begin{small}
		\[\min_{y\in\mathbb{C}^{N_{\ell+1}}}\|x_{N_{\ell}}-P_{N_{\ell},N_{\ell+1}}y\|_{2}^{2}\leq {\kappa(A_{N_\ell},P_{{N_\ell,{N_{\ell+1}}}})}\|x_{N_{\ell}}\|_{A_{N_{\ell}}}^{2},\quad \forall x_{N_{\ell}}\in\mathbb{C}^{N_{\ell}}.\]
	\end{small}
	\end{itemize}
	Then $ {\kappa(A_{N_\ell},P_{{N_\ell,{N_{\ell+1}}}})}\geq\sigma_{\rm{post}}$ and
\begin{small}
	\begin{align*}
	\|{\rm TGM}(A_{N_{\ell}},I_{N_{\ell}},V_{N_{\ell},\rm{post}}^{\nu_{\rm{post}}},P_{N_{\ell},N_{\ell+1}})\|_{A_{N_{\ell}}}\leq\sqrt{1-\sigma_{\rm{post}}/ {\kappa(A_{N_\ell},P_{{N_\ell,{N_{\ell+1}}}})}}.
	\end{align*}
\end{small}
\end{Theorem}

\subsection{Multigrid methods for block circulant and block Toeplitz systems}\label{ssec:multigrid_for_block}

If the coefficient matrix is of the form $\mathcal{C}_{\mathbf{n}}(\mathbf{f})$ or $T_{\mathbf{n}}(\mathbf{f})$, with $\mathbf{f}$ being a matrix-valued trigonometric polynomial, $\mathbf{f}\geq0$, sufficient conditions for the convergence and optimality of the TGM can be expressed in relation to $\mathbf{f}$. 


In particular, \cite{block_multigrid2} {provides} sufficient conditions that should be fulfilled by the grid transfer operators in order to validate the approximation property (b) of Theorem \ref{thm:teoconv}. In the following we report a brief description of the result \cite{block_multigrid2} in the general multilevel block Toeplitz setting.

 Suppose $\mathbf{f}$ is a $s\times s$ matrix-valued trigonometric polynomial, $\mathbf{f}\geq0$, such that there exist unique $\boldsymbol{\theta}_0\in[0,2\pi)^d$ and $\bar{\jmath} \in\{1,\dots,s\}$ such that 
\begin{equation}\label{eqn:condition_on_f_multi}
\left\{\begin{array}{ll}
\lambda_j(\mathbf{f}(\boldsymbol{\theta}))=0 & \mbox{for } \boldsymbol{\theta}=\boldsymbol{\theta}_0 \mbox{ and } j=\bar{\jmath}, \\
\lambda_j(\mathbf{f}(\boldsymbol{\theta}))>0 & {\rm otherwise}.
\end{array}\right.
\end{equation}
The latter means that only one eigenvalue function of $\mathbf{f}$ has exactly one zero in $ \boldsymbol{\theta}_0$ and $\mathbf{f}$ is positive definite in $[0,2\pi)^{ {d}}\backslash\{ \boldsymbol{\theta}_0\}$.
Define  $q_{\bar{\jmath}}(\boldsymbol{\theta}_0)$ as the eigenvector function of $\mathbf{f}(\boldsymbol{\theta}_0)$ associated with $\lambda_{\bar{\jmath}}(\mathbf{f}(\boldsymbol{\theta}_0))=0$. 
Moreover, define $\Omega(\boldsymbol{\theta})= \left\{\boldsymbol{\theta}+ \pi \boldsymbol{\eta}, \, \boldsymbol{\eta}\in \{0, 1\}^{ {d}} \right\}$.
Let us define a matrix-valued trigonometric polynomial  $\mathbf{p}(\cdot)$ such that
\begin{itemize}
\item \begin{equation}\label{eqn:condition_on_p_multi_intr}
	\sum_{\boldsymbol{\xi} \in \Omega(\boldsymbol{\theta}) }\mathbf{p}(\boldsymbol{\xi} )^{H}\mathbf{p}(\boldsymbol{\xi} )>0, \quad\forall\, \boldsymbol{\theta} \in[0,2\pi)^{ {d}},
\end{equation}
 which implies that the trigonometric function
\begin{equation*}
\mathbf{s}(\boldsymbol{\theta}) = \mathbf{p}(\boldsymbol{\theta})\left(\sum_{\boldsymbol{\xi}  \in \Omega(\boldsymbol{\theta}) }\mathbf{p}(\boldsymbol{\xi} )^{H}\mathbf{p}(\boldsymbol{\xi} )\right)^{-1}\mathbf{p}(\boldsymbol{\theta})^{H}
\end{equation*}
is well-defined for all $\boldsymbol{\theta} \in[0,2\pi)^{ {d}}$.
\item \begin{equation}\label{eqn:condition_on_s_multi_intr}
	\mathbf{s}(\boldsymbol{\theta} _0)q_{\bar{\jmath}}(\boldsymbol{\theta}_0) = q_{\bar{\jmath}}(\boldsymbol{\theta}_0).
\end{equation}
\item \begin{equation}\label{eqn:condition_on_s_f_multi_intr}
\limsup_{\boldsymbol{\theta} \rightarrow \boldsymbol{\theta}_0} \lambda_{\bar{\jmath}}(\mathbf{f}(\boldsymbol{\theta}))^{-1}(1-\lambda_{\bar{\jmath}}(\mathbf{s}(\boldsymbol{\theta})))=c,
\end{equation}
 where $c\in\mathbb{R}$ is a constant. 
\end{itemize}
In order to have optimal convergence of the TGM it is sufficient compute a grid transfer operator as
\begin{equation}
  P_{N_\ell,N_{\ell+1}}=T_{\mathbf{n}_\ell}(\mathbf{p}) \left(K_{N_\ell,N_{\ell+1}}\otimes I_{s}\right),
 \end{equation} 
  where $K_{N_\ell,N_{\ell+1}}$ is a $N_\ell\times N_{\ell+1}$  multilevel cutting matrix obtained as the Kronecker product of proper unilevel cutting matrices and $T_{\mathbf{n}_\ell}(\mathbf{p})$ is a multilevel block-Toeplitz matrix generated by $\mathbf{p}$.
{
\subsection{TGM for saddle-point systems}
\label{ssec:saddle_uni_scalar}
In the previous section we focused on the case where the coefficient matrix is a block structured and non-negative definite matrix. For matrices in saddle-point form, R\"{u}ge-Stuben theory and Theorem \ref{thm:teoconv}
cannot be directly exploited. In \cite{MR3439215} a general TGM procedure has been presented for a system of the form
\begin{equation}\label{eq:saddle_notay}
\mathcal{A}=	\begin{bmatrix}
		A & B^T \\
		B & -C
	\end{bmatrix}
\end{equation}
where $A$ is an $n\times n$ HPD matrix, $C$ is an $\hat{n}\times \hat{n}$ non-negative definite matrix and $B$ is such that its rank is $\hat{n}\leq n$ or that $C$ is positive definite on the null space of $B^T$. 
Indeed, we have the following result.
\begin{Theorem}\cite[Theorem 4.4]{MR3439215}\label{thm:teo_notay}
Let $\mathcal{A}$ be as in (\ref{eq:saddle_notay}).
 { Define $D_A=\diag(A)$ and let} 
$\alpha$ be a positive number such that $\alpha<2(\lambda_{\max}(D_A^{-1}A))^{-1}$. {Compute} 
\begin{equation}\label{eq:LAU}
\hat{\mathcal{A}}=\mathcal{LAU}, \qquad
\mathcal{L}= \begin{bmatrix}
I_{n}\\
\alpha B D_{A}^{-1}&-I_{\hat{n}}
\end{bmatrix}, \qquad \mathcal{U}=\begin{bmatrix}
I_{n}&-\alpha D_{A}^{-1}B^T\\
&I_{\hat{n}}
\end{bmatrix},
\end{equation}
 {and define
 \begin{equation}\label{eq:Chat}
  \hat{C}=C+B(2\alpha D_A^{-1}-\alpha^2D_A^{-1}AD_A^{-1})B^T, \quad D_{\hat{C}}={\rm diag} (\hat{C}).
 \end{equation} 
 } 
Let $P_A$ and $P_{\hat{C}}$ be, respectively, $n\times k$ and $\hat{n}\times \ell$ matrices of rank $k<n$ and $\ell<\hat{n}$.  {Define the prolongation operator}

\begin{equation}\label{eq:PNotay}
\mathcal{P}=\begin{bmatrix}
P_{A}& \\
& P_{\hat{C}}
\end{bmatrix}
\end{equation}
for the global system involving $\mathcal{\hat{A}}$. Suppose that  the pairs $(A,P_A)$ and $(\hat{C}, P_{\hat{C}})$ fulfill the approximation property {  in Theorem \ref{thm:teoconv} with associated
approximation property constants  ${\kappa}({A},P_{A})$ and  ${\kappa}({\hat{C}},P_{\hat{C}})$}.
  Then, the spectral radius of the TGM iteration matrix {using} one iteration of the damped Jacobi method with relaxation parameter $\omega$ {as  post smoother} satisfies
\begin{equation}\label{eq:notay_max}
\rho({\rm TGM}(\mathcal{\hat{A}},\mathcal{P},\omega))\le \max\left(1-\frac{\omega}{{\kappa}({A},P_{A})},  \, 1-\frac{\omega}{{\kappa}({\hat{C}},P_{\hat{C}})}, \, \omega \hat{\gamma}_A-1,  \, \omega \hat{\gamma}_{\hat{C}}-1,  \, \sqrt{1-\frac{\omega(2-\omega\tilde{\gamma})}{\tilde{\kappa}}} \right),
\end{equation}
where \begin{align*}
\hat{\gamma}_{A}= \left(\alpha\left(2-\alpha\lambda_{\max}(D_{A}^{-1}A)\right)\right)^{-1},& \qquad \hat{\gamma}_{\hat{C}}= \lambda_{\max}\left(D_{\hat{C}}^{-1}(C+BA^{-1}B^T)\right),  \\
\tilde{\gamma}=\frac{2\hat{\gamma}_{A}\hat{\gamma}_{\hat{C}}}{\hat{\gamma}_{A}+\hat{\gamma}_{\hat{C}}},& \qquad \tilde{\kappa}=\frac{2{\kappa}({A},P_{A}){\kappa}({\hat{C}},P_{\hat{C}})}{{\kappa}({A},P_{A})+{\kappa}({\hat{C}},P_{\hat{C}})}.
\end{align*}
\end{Theorem}
}

\section{TGM for saddle-point systems in the multilevel block setting}
\label{sec:saddle_multi_block}
The aim of the current section is to provide the theoretical background for the development and analysis of a TGM for the Stokes problem discretized with the technique of Subsection \ref{ssec:Q1_iso_Q1_matthias}. The starting point is Theorem \ref{thm:teo_notay}, which we want to combine with the theory for multilevel block structured matrices of Subsection \ref{ssec:multigrid_for_block}. The derivations are an extension of the results in \cite{saddle_DBFF}, where {we} consider a general unilevel circulant setting and a multilevel circulant setting which arises from the discretization of Stokes-like equations. The generalization that we consider in what follows consists in dealing with matrix-valued symbols instead of scalar-valued ones and this forces us to give the conditions on the symbols in a different form, owing to the non-commutativity of the matrix-matrix product.
 
The following results concern the properties of a matrix of the form
\begin{equation*}
\mathcal{A}=	\begin{bmatrix}
		\mathcal{C}_\mathbf{n}(\mathbf{f}_{A_x}) & O & \mathcal{C}_\mathbf{n}^T(\mathbf{f}_{B_x}) \\
		O & \mathcal{C}_\mathbf{n}(\mathbf{f}_{A_y}) & \mathcal{C}_\mathbf{n}^T(\mathbf{f}_{B_y}) \\
		\mathcal{C}_\mathbf{n}(\mathbf{f}_{B_x}) & \mathcal{C}_\mathbf{n}(\mathbf{f}_{B_y}) & -\mathcal{C}_\mathbf{n}(\mathbf{f}_{C})
	\end{bmatrix}\in \mathbb{R}^{(2s_a+s_c)n^2\times (2s_a+s_c)n^2}
\end{equation*}
with multilevel block circulant blocks generated by multivariate matrix-valued functions. For the sake of readability, we focus on the case $d=2$, $\mathbf{n}=(n,n)$ and, moreover, we fix $s_a\ge s_c$. An extension to larger dimensional problems is straightforward, in this case additional blocks will be part of $\mathcal{A}$, e.g., for $d=3$ we will have an additional block $A_z$ on the block diagonal and additionally $B_z$ and $B_z^T$ in the last block row and column, respectively.
\begin{Lemma}\label{lem:alpha_multi_block}
Let $\tilde{A}=\begin{bmatrix}
A_x&O\\
O & A_y
\end{bmatrix}\in \mathbb{R}^{2sn^2\times 2 s n^2}$, $A_x=\mathcal{C}_{\mathbf{n}}(\mathbf{f}_{A_x})$, $A_y=\mathcal{C}_{\mathbf{n}}(\mathbf{f}_{A_y})$ with $\mathbf{f}_{A_x}$, $\mathbf{f}_{A_y}$ $s\times s$ HPD matrix-valued bi-variate trigonometric polynomials.

If $\alpha$ is a positive number such that
\[\alpha< 2 \left(\max_{\zeta\in\{x,y\}} \left\{\max_{j=1,\dots,s} \left\{\left(\hat{a}_0^{(j,j)}(\mathbf{f}_{A_{\zeta}})\right)^{-1}\right\}	\left\| \mathbf{f}_{A_{\zeta}} \right\|_{\infty}\right\}\right)^{-1}
\]
then the matrix $2\alpha D^{-1}_{\tilde{A}}-\alpha^2D^{-1}_{\tilde{A}} \tilde{A} D^{-1}_{\tilde{A}} $ is HPD, where $D_{\tilde{A}}=\diag(\tilde{A})$.  
\end{Lemma}
\begin{proof}
Since $\tilde{A}$ is HPD, then $D_{\tilde{A}}$ is HPD. Hence, writing
\[
2\alpha D^{-1}_{\tilde{A}}-\alpha^2D^{-1}_{\tilde{A}} \tilde{A} D^{-1}_{\tilde{A}} = 
D^{-\frac{1}{2}}_{\tilde{A}}\left(2\alpha I_s-\alpha^2D^{-\frac{1}{2}}_{\tilde{A}} \tilde{A} D^{-\frac{1}{2}}_{\tilde{A}}\right)D^{-\frac{1}{2}}_{\tilde{A}}
\]
we see that we need to prove that the Hermitian matrix $2\alpha I_s-\alpha^2D^{-\frac{1}{2}}_{\tilde{A}} \tilde{A} D^{-\frac{1}{2}}_{\tilde{A}}$ is HPD. The eigenvalues of $2\alpha I_s-\alpha^2D^{-\frac{1}{2}}_{\tilde{A}} \tilde{A} D^{-\frac{1}{2}}_{\tilde{A}}$ are of the form
$2\alpha -\alpha^2 \lambda_j \left(D^{-\frac{1}{2}}_{\tilde{A}} \tilde{A} D^{-\frac{1}{2}}_{\tilde{A}}\right)$,
where $\lambda_j \left(D^{-\frac{1}{2}}_{\tilde{A}} \tilde{A} D^{-\frac{1}{2}}_{\tilde{A}}\right)$ are the positive eigenvalues of $D^{-\frac{1}{2}}_{\tilde{A}} \tilde{A} D^{-\frac{1}{2}}_{\tilde{A}}$. Recalling that $\alpha>0$ by assumption, we have
\[
	\lambda_{\min}\left(\alpha I_s-\alpha^2D^{-\frac{1}{2}}_{\tilde{A}} \tilde{A} D^{-\frac{1}{2}}_{\tilde{A}}\right)
	= 2\alpha -\alpha^2 \lambda_{\max} \left(D^{-\frac{1}{2}}_{\tilde{A}} \tilde{A} D^{-\frac{1}{2}}_{\tilde{A}}\right),
\]
which implies that $\alpha I_s-\alpha^2D^{-\frac{1}{2}}_{\tilde{A}} \tilde{A} D^{-\frac{1}{2}}_{\tilde{A}}$ is HPD if
\begin{equation}\label{eq:alpha_matform}
	0<\alpha< \frac{2}{\lambda_{\max} \left(D^{-\frac{1}{2}}_{\tilde{A}} \tilde{A} D^{-\frac{1}{2}}_{\tilde{A}}\right)}.
\end{equation}

By construction,  $D^{-1}_{\tilde{A}}=\begin{bmatrix}
D^{-1}_{A_x}&O\\
O & D^{-1}_{A_y}
\end{bmatrix}$, with 
\[D^{-1}_{A_x}=I_{n^2}\otimes \diag_{1\le j \le s} \left(\hat{a}_0^{(j,j)}(\mathbf{f}_{A_x})\right)^{-1}
, \quad 
D^{-1}_{A_y}=I_{n^2}\otimes \diag_{1\le j \le s} \left(\hat{a}_0^{(j,j)}(\mathbf{f}_{A_y})\right)^{-1},
\]
where $\hat{a}_0^{(1,1)}(\mathbf{f}_{A_x}), \dots, \hat{a}_0^{(s,s)}(\mathbf{f}_{A_x})$ and $\hat{a}_0^{(1,1)}(\mathbf{f}_{A_y}), \dots, \hat{a}_0^{(s,s)}(\mathbf{f}_{A_y})$ are the positive diagonal entries of the 0th block diagonal Fourier coefficient of $\mathbf{f}_{A_x}$ and $\mathbf{f}_{A_y}$, respectively. Consequently, we write
\begin{align*}
	\lambda_{\max} &\left(D^{-\frac{1}{2}}_{\tilde{A}} \tilde{A} D^{-\frac{1}{2}}_{\tilde{A}}\right) =\\
	& = \max_{\zeta\in\{x,y\}} \lambda_{\max}\left(D^{-\frac{1}{2}}_{A_{\zeta}} \tilde{A} D^{-\frac{1}{2}}_{A_{\zeta}} \right)\\
	& = \max_{\zeta\in\{x,y\}} \lambda_{\max}\left(\mathcal{C}_\mathbf{n}\left(
									\diag_{1\le j \le s} \left(\hat{a}_0^{(j,j)}(\mathbf{f}_{A_{\zeta}})\right)^{-\frac{1}{2}}
									\mathbf{f}_{A_{\zeta}}
									\diag_{1\le j \le s} \left(\hat{a}_0^{(j,j)}(\mathbf{f}_{A_{\zeta}})\right)^{-\frac{1}{2}}\right)\right)\\
	& \le \max_{\zeta\in\{x,y\}} \left\|	\diag_{1\le j \le s} \left(\hat{a}_0^{(j,j)}(\mathbf{f}_{A_{\zeta}})\right)^{-\frac{1}{2}}
									\mathbf{f}_{A_{\zeta}}
									\diag_{1\le j \le s} \left(\hat{a}_0^{(j,j)}(\mathbf{f}_{A_{\zeta}})\right)^{-\frac{1}{2}}\right\|_{\infty}\\
	& \le \max_{\zeta\in\{x,y\}} \left\{\max_{j=1,\dots,s} \left\{\left(\hat{a}_0^{(j,j)}(\mathbf{f}_{A_{\zeta}})\right)^{-1}\right\}	\left\| \mathbf{f}_{A_{\zeta}} \right\|_{\infty}\right\}.
\end{align*}
Hence, if \[0<\alpha< \frac{2}{\max_{\zeta\in\{x,y\}} \left\{\max_{j=1,\dots,s} \left\{\left(\hat{a}_0^{(j,j)}(\mathbf{f}_{A_{\zeta}})\right)^{-1}\right\}	\left\| \mathbf{f}_{A_{\zeta}} \right\|_{\infty}\right\}}
\]
the parameter $\alpha$ fulfils condition (\ref{eq:alpha_matform}) and the thesis of the Lemma follows.
\end{proof}

\begin{Remark}\label{rmk:Chat}
	In order to follow the procedure analysed by Theorem \ref{thm:teo_notay}, we need to study the matrix $\hat{C}$ defined in \eqref{eq:Chat}. Theorem \ref{thm:teo_notay} requires that either $B$ is full-rank or $C$ is HPD on the null space of $B^T$. In the next theorem, we will require that $\hat{C}$ is HPD, but this implies that the latter condition on $B$ and $C$ is fulfilled.	
	Indeed, if $\hat{C}$ is HPD and there exists $v\in\mathbb{R}^{s_cn^2}$ such that $v^TB=0$, then
	\begin{align*}
		0< v^T\hat{C}v 
		<v^TCv+v^TB\left(2\alpha D_A^{-1}-\alpha^2D_A^{-1}AD_A^{-1}\right)B^Tv
		<v^T{C}v.
	\end{align*}	
	Conversely, note that if $B$ is full-rank or $C$ is HPD on the null space of $B^T$, then $\hat{C}$ is HPD as long as $\alpha$ is chosen according to Lemma \ref{lem:alpha_multi_block}.
\end{Remark}

We are now ready to state the main theorem concerning the convergence of a TGM of the form of Theorem \ref{thm:teo_notay} in the multilevel block circulant setting.
\begin{Theorem}\label{thm:Notay_symbol_2D_block}
Let us define the matrix 
\begin{equation}\label{eq:problem_matrix}
\mathcal{A}=	\begin{bmatrix}
		\tilde{A} & B^T \\
		B & -C
	\end{bmatrix}\in \mathbb{R}^{(2s_a+s_c)n^2\times (2s_a+s_c)n^2}
\end{equation}
having the blocks with the following structure:
\begin{itemize}
\item $\tilde{A}=\begin{bmatrix}
A_x&O\\
O & A_y
\end{bmatrix}$ where $A_x=\mathcal{C}_{\mathbf{n}}(\mathbf{f}_{A_x})$, $A_y=\mathcal{C}_{\mathbf{n}}(\mathbf{f}_{A_y})$ with $\mathbf{f}_{A_x}$, $\mathbf{f}_{A_y}$ being $s_a\times s_a$ matrix-valued bivariate trigonometric polynomials. 
\item  $C=\mathcal{C}_{\mathbf{n}}(\mathbf{f}_C)$ with $\mathbf{f}_C$ $s_c\times s_c$ non-negative bivariate trigonometric polynomial.
\item $B=[B_x,B_y]$ where $B_x=\mathcal{C}_{\mathbf{n}}(\mathbf{f}_{B_x})$, $B_y=\mathcal{C}_{\mathbf{n}}(\mathbf{f}_{B_y})$ with $\mathbf{f}_{B_x}$ and $\mathbf{f}_{B_y}$ being $s_a\times s_c$ matrix-valued bivariate trigonometric polynomials such that
\begin{equation}\label{eq:bound_for_schur}
\mathbf{f}_{B_x}\left(\mathbf{f}_{A_x}\right)^{-1}\mathbf{f}_{B_x}^H+\mathbf{f}_{B_y}\left(\mathbf{f}_{A_y}\right)^{-1}\mathbf{f}_{B_y}^H \in L^{\infty}\left([-\pi,\pi]^2,s_c\right).
\end{equation}
%
\end{itemize}
Let $\alpha$ be a positive number such that 
\begin{equation}\label{eq:choice_of_alpha}
	\alpha< 2 \left(\max_{\zeta\in\{x,y\}} \left\{\max_{j=1,\dots,s} \left\{\left(\hat{a}_0^{(j,j)}(\mathbf{f}_{A_{\zeta}})\right)^{-1}\right\}	\left\| \mathbf{f}_{A_{\zeta}} \right\|_{\infty}\right\}\right)^{-1}
\end{equation}
and define $\mathcal{\hat{A}}=\mathcal{LAU}$ where
\begin{equation}\label{eq:def_LAU_2D_block}
\mathcal{L}= \begin{bmatrix}
I_{2s_an^2}\\
\alpha B D_{\tilde{A}}^{-1}&-I_{s_cn^2}
\end{bmatrix}, \quad \mathcal{U}=\begin{bmatrix}
I_{2s_an^2}&-\alpha D_{\tilde{A}}^{-1}B^T\\
&I_{s_cn^2}
\end{bmatrix}.
\end{equation}
Let $P_{\tilde{A}}$ and $P_{\hat{C}}$ be the full rank matrices defining the prolongation $\mathcal{P}$ for the global system $\mathcal{\hat{A}}$. That is,
\begin{equation*}
\mathcal{P}=\begin{bmatrix}
P_{\tilde{A}}& \\
& P_{\hat{C}}
\end{bmatrix},
\end{equation*}
  where $\hat{C}=C+B\left(2\alpha D_{\tilde{A}^{-1}}-\alpha^2D_{\tilde{A}^{-1}}\tilde{A}D_{\tilde{A}^{-1}}\right)B^T=C_{\mathbf{n}}(\mathbf{f}_{\hat{C}})$, $$P_{\tilde{A}}=\begin{bmatrix}
  \mathcal{C}_{\mathbf{n}}(\mathbf{p}_{A_x})(K_{\mathbf{n}}^T\otimes I_{s_a}) & \\
   &   \mathcal{C}_{\mathbf{n}}(\mathbf{p}_{A_y})(K_{\mathbf{n}}^T \otimes I_{s_a})
  \end{bmatrix}, \quad P_{\hat{C}}=\mathcal{C}_\mathbf{n}(\mathbf{p}_{\hat{C}})(K_\mathbf{n}^T\otimes I_{s_c})$$ with $\mathbf{p}_{A_x}$, $\mathbf{p}_{A_y}$, and $\mathbf{p}_{\hat{C}}$ bi-variate matrix-valued trigonometric polynomials.
%
Consider a TGM associated with a basic scheme using only a single step of damped Jacobi post-smoothing with relaxation parameter $\omega$. If
\begin{enumerate}
\item the trigonometric polynomials $\mathbf{f}_{A_x}$, $\mathbf{f}_{A_y}$, and $\mathbf{f}_{\hat{C}}$ are in the setting of (\ref{eqn:condition_on_f_multi}) and they are full rank at all grid points defined in (\ref{eq:grid});
\item the pairs $(\mathbf{p}_{A_x},\mathbf{f}_{A_x})$, $(\mathbf{p}_{A_y},\mathbf{f}_{A_y})$, and $(\mathbf{p}_{\hat{C}}, \mathbf{f}_{\hat{C}})$ fulfil relations (\ref{eqn:condition_on_p_multi_intr})--(\ref{eqn:condition_on_s_f_multi_intr});
 \item the relaxation parameter $\omega$ is a positive number such that 
\begin{multline}\label{eq:choice_of_omega}
 \omega<2\min \left\{2\alpha -\alpha^2\max_j\left\{ \frac{\|\mathbf{f}_{A_x}\|_\infty}{{\hat{a}_0^{(j,j)}(\mathbf{f}_{A_x})}}, \frac{\|\mathbf{f}_{A_y}\|_\infty}{{\hat{a}_0^{(j,j)}(\mathbf{f}_{A_y})}}\right\},\right.\\
 \left.\max_j\hat{a}_0^{(j,j)}(\mathbf{f}_{\hat{C}})\left\|\mathbf{f}_C+\mathbf{f}_{B_x}\mathbf{f}_{A_x}^{-1}\mathbf{f}_{B_x}^H+\mathbf{f}_{B_y}\mathbf{f}_{A_y}^{-1}\mathbf{f}_{B_y}^H\right\|_\infty^{-1}\right\},
\end{multline} 
\end{enumerate}
then $\rho({\rm TGM}(\mathcal{\hat{A}},\mathcal{P},\omega))<1.$
\end{Theorem} 
\begin{proof}
As a consequence of Remark \ref{rmk:fourier-vs-funzione}, the assumptions on $\mathbf{f}_{A_x}$, $\mathbf{f}_{A_y}$ and $\mathbf{f}_{\hat{C}}$ imply that $\tilde{A}$ is HPD. Moreover, we see from Remark \ref{rmk:Chat} that either $B$ is full-rank or $C$ is HPD on the null space of $B^T$.

The assumptions on $\mathbf{p}_{A_x}$, $\mathbf{p}_{A_y}$, $\mathbf{p}_{\hat{C}}$ and $\mathbf{f}_{A_x}$, $\mathbf{f}_{A_y}$, $\mathbf{f}_{\hat{C}}$ ensure the validation of the approximation property by \cite[Section 6]{block_multigrid2}. That is, the pairs $(\tilde{A}, P_{\tilde{A}})$ and $({\hat{C}},P_{\hat{C}})$ fulfil the approximation property in the 2D block setting. 
Therefore, we can apply Theorem \ref{thm:teo_notay} and by \eqref{eq:notay_max} we have
$$\rho({\rm TGM}(\mathcal{\hat{A}},\mathcal{P},\omega))\le \max\left(1-\frac{\omega}{{\kappa}(\tilde{A},P_{\tilde{A}})}, 1-\frac{\omega}{{\kappa}({\hat{C}},P_{\hat{C}})}, \omega \hat{\gamma}_{\tilde{A}}-1, \omega \hat{\gamma}_{\hat{C}}-1, \sqrt{1-\frac{\omega(2-\omega\tilde{\gamma})}{\tilde{\kappa}}} \right).$$
In conclusion, in order to prove that $\rho({\rm TGM}(\mathcal{\hat{A}},\mathcal{P},\omega))<1$, we need to prove that each  of the quantities
in the maximum is bounded from above by a constant strictly smaller that 1.
 
 The quantities $1-\frac{\omega}{{\kappa}({A},P_{A})}$ and $1-\frac{\omega}{{\kappa}({\hat{C}},P_{\hat{C}})}$ are strictly smaller than 1 because the approximation property constants ${{\kappa}({A},P_{A})}$ and ${{\kappa}({\hat{C}},P_{\hat{C}})}$  are finite and positive.
 
Concerning the two terms $\omega \hat{\gamma}_{\tilde{A}}-1$ and  $\omega \hat{\gamma}_{\hat{C}}-1$ in the maximum, we estimate
\begin{equation}\label{eq:gamma1_block}
 \begin{split}
 \hat{\gamma}_{\tilde{A}}&= \left(\alpha\left(2-\alpha\lambda_{\max}(D_{\tilde{A}}^{-1}\tilde{A})\right)\right)^{-1}
 =\\
 &\left(2\alpha -\alpha^2\lambda_{\max}\left(\begin{bmatrix}
D^{-1}_{A_x}A_x&O\\
O & D^{-1}_{A_y}A_y
\end{bmatrix}	\right)\right)^{-1}
 \le\\
 & \left(2\alpha -\alpha^2\max_j\left\{ \frac{\|\mathbf{f}_{A_x}\|_\infty}{\hat{a}_0^{(j,j)}(\mathbf{f}_{A_x})}, \frac{\|\mathbf{f}_{A_y}\|_\infty}{{\hat{a}_0^{(j,j)}(\mathbf{f}_{A_y})}}\right\}\right)^{-1} , 
 \end{split}
 \end{equation}  
   \begin{equation}\label{eq:gamma2_block}
   \begin{split}
\hat{\gamma}_{\hat{C}}= &\lambda_{\max}\left(D_{\hat{C}}^{-1}(C+B\tilde{A}^{-1}B^T)\right)
\le  \max_{j}\frac{1}{{\hat{a}_0^{(j,j)}(\mathbf{f}_{\hat{C}})}} \lambda_{\max}\left(\mathcal{C}_n\left(\mathbf{f}_C+\mathbf{f}_{B_x}\mathbf{f}_{A_x}^{-1}\mathbf{f}_{B_x}^H+\mathbf{f}_{B_y}\mathbf{f}_{A_y}^{-1}\mathbf{f}_{B_y}^H\right)\right)\\
\le 
&\max_{j} \frac{1}{{\hat{a}_0^{(j,j)}(\mathbf{f}_{\hat{C}})}} \left\|\mathbf{f}_C+\mathbf{f}_{B_x}\mathbf{f}_{A_x}^{-1}\mathbf{f}_{B_x}^H+\mathbf{f}_{B_y}\mathbf{f}_{A_y}^{-1}\mathbf{f}_{B_y}^H\right\|_\infty, 
\end{split}
\end{equation}
where in the latter inequality we are using assumption (\ref{eq:bound_for_schur}) and the sub-additivity of norms. Hence with the choice of $\omega$ in assumption 3., we have
\begin{align*}
\omega \hat{\gamma}_{\tilde{A}}-1<&2 \left(2\alpha -\alpha^2\max_j\left\{ \frac{\|\mathbf{f}_{A_x}\|_\infty}{{\hat{a}_0^{(j,j)}(\mathbf{f}_{A_x})}}, \frac{\|\mathbf{f}_{A_y}\|_\infty}{{\hat{a}_0^{(j,j)}(\mathbf{f}_{A_y})}}\right\}\right)\cdot\\
&\cdot \left(2\alpha -\alpha^2\max_j\left\{ \frac{\|\mathbf{f}_{A_x}\|_\infty}{\hat{a}_0^{(j,j)}(\mathbf{f}_{A_x})}, \frac{\|\mathbf{f}_{A_y}\|_\infty}{{\hat{a}_0^{(j,j)}(\mathbf{f}_{A_y})}}\right\}\right)^{-1} -1\\
=&1,
\end{align*}
and
\begin{align*}
\omega \hat{\gamma}_{\hat{C}}-1<& 2\max_j\hat{a}_0^{(j,j)}(\mathbf{f}_{\hat{C}})\left\|\mathbf{f}_C+\mathbf{f}_{B_x}\mathbf{f}_{A_x}^{-1}\mathbf{f}_{B_x}^H+\mathbf{f}_{B_y}\mathbf{f}_{A_y}^{-1}\mathbf{f}_{B_y}^H\right\|_\infty^{-1} \cdot\\
&\cdot \max_{j} \frac{1}{{\hat{a}_0^{(j,j)}(\mathbf{f}_{\hat{C}})}} \left\|\mathbf{f}_C+\mathbf{f}_{B_x}\mathbf{f}_{A_x}^{-1}\mathbf{f}_{B_x}^H+\mathbf{f}_{B_y}\mathbf{f}_{A_y}^{-1}\mathbf{f}_{B_y}^H\right\|_\infty -1\\
=&1.
\end{align*}
Finally, in order to prove that $\sqrt{1-\frac{\omega(2-\omega\tilde{\gamma})}{\tilde{\kappa}}} <1$, we prove that $2-\omega \tilde{\gamma}>0$.
From hypothesis 3., we have
\begin{equation}\label{eq:om1_block}
\begin{split}
	\omega
	<&  2\alpha -\alpha^2\max_j\left\{ \frac{\|\mathbf{f}_{A_x}\|_\infty}{{\hat{a}_0^{(j,j)}(\mathbf{f}_{A_x})}}, \frac{\|\mathbf{f}_{A_y}\|_\infty}{{\hat{a}_0^{(j,j)}(\mathbf{f}_{A_y})}}\right\}+\\
 &+\max_j\hat{a}_0^{(j,j)}(\mathbf{f}_{\hat{C}})\left\|\mathbf{f}_C+\mathbf{f}_{B_x}\mathbf{f}_{A_x}^{-1}\mathbf{f}_{B_x}^H+\mathbf{f}_{B_y}\mathbf{f}_{A_y}^{-1}\mathbf{f}_{B_y}^H\right\|_\infty^{-1}.
 \end{split}
\end{equation}
Moreover, using the estimations \eqref{eq:gamma1_block} and \eqref{eq:gamma2_block}, we have
\begin{equation}\label{eq:gamma3_block}
\begin{split}
\tilde{\gamma}&=\frac{2\hat{\gamma}_{A}\hat{\gamma}_{\hat{C}}}{\hat{\gamma}_{A}+\hat{\gamma}_{\hat{C}}}\\
&\le 
2 \left(
\frac{\left(2\alpha -\alpha^2\max_j\left\{ \frac{\|\mathbf{f}_{A_x}\|_\infty}{\hat{a}_0^{(j,j)}(\mathbf{f}_{A_x})}, \frac{\|\mathbf{f}_{A_y}\|_\infty}{{\hat{a}_0^{(j,j)}(\mathbf{f}_{A_y})}}\right\}\right)^{-1} 
\max_{j} \frac{\left\|\mathbf{f}_C+\mathbf{f}_{B_x}\mathbf{f}_{A_x}^{-1}\mathbf{f}_{B_x}^H+\mathbf{f}_{B_y}\mathbf{f}_{A_y}^{-1}\mathbf{f}_{B_y}^H\right\|_\infty}{{\hat{a}_0^{(j,j)}(\mathbf{f}_{\hat{C}})}} }
{\left(2\alpha -\alpha^2\max_j\left\{ \frac{\|\mathbf{f}_{A_x}\|_\infty}{\hat{a}_0^{(j,j)}(\mathbf{f}_{A_x})}, \frac{\|\mathbf{f}_{A_y}\|_\infty}{{\hat{a}_0^{(j,j)}(\mathbf{f}_{A_y})}}\right\}\right)^{-1}  
+\max_{j} \frac{\left\|\mathbf{f}_C+\mathbf{f}_{B_x}\mathbf{f}_{A_x}^{-1}\mathbf{f}_{B_x}^H+\mathbf{f}_{B_y}\mathbf{f}_{A_y}^{-1}\mathbf{f}_{B_y}^H\right\|_\infty}{{\hat{a}_0^{(j,j)}(\mathbf{f}_{\hat{C}})}} }
\right),
\end{split}
\end{equation}
where we used the fact that the function $(x,y)\mapsto 2 \frac{x y}{x+y} $ is increasing in $(0,+\infty)\times (0,+\infty)$ with respect to the component-wise partial ordering in $\mathbb{R}^2$.
Then, combining \eqref{eq:om1_block} and \eqref{eq:gamma3_block}, we obtain $\omega\tilde{\gamma}< 2$.
\end{proof}

\begin{Remark}\label{rmk:results_Czero_sc1}
	In the previous theorem, some of the assumptions concern the pair $\left({\hat{C}},P_{\hat{C}}\right)$. The matrix $\hat{C}$ is computed from $A$, $B$ and $C$, so the required assumptions can be stated involving explicitly the latter three matrices. For the sake of brevity, we discuss this point here in a particular setting, which is relevant from the application point of view. Indeed, we focus on the case where the matrix $\mathcal{A}$ in \eqref{eq:problem_matrix} is such that $C$ is the null matrix and $s_C = 1$, so that ${f}_{\hat{C}}$ is the scalar--valued function
	\[
		{f}_{\hat{C}}=\mathbf{f}_{B_x}\mathbf{g}_x\mathbf{f}_{B_x}^H+\mathbf{f}_{B_y}\mathbf{g}_y\mathbf{f}_{B_y}^H,
	\]
	where we denoted by $\mathbf{g}_{\zeta}$ the generating function of $2\alpha D^{-1}_{{A}_{\zeta}}-\alpha^2D^{-1}_{{A_{\zeta}}} {A_{\zeta}} D^{-1}_{{A}_{\zeta}} $, for ${\zeta}\in\{x,y\}$. If we choose $\alpha$ according to \eqref{eq:choice_of_alpha}, then the functions $\mathbf{g}_{\zeta}$ are HPD at all points. Assumption 1. of Theorem \ref{thm:Notay_symbol_2D_block} then corresponds  to requiring that there exists a $\boldsymbol{\theta}_0$ such that $\mathbf{f}_{B_x}(\boldsymbol{\theta})$ and $\mathbf{f}_{B_y}(\boldsymbol{\theta})$ are both the null vectors if and only if $\boldsymbol{\theta}=\boldsymbol{\theta}_0$. Regarding the choice of ${p}_{\hat{C}}$, it is straightforward to see that in this scalar setting relations (\ref{eqn:condition_on_p_multi_intr})--(\ref{eqn:condition_on_s_f_multi_intr}) are fulfilled by the pair $({p}_{\hat{C}}, {f}_{\hat{C}})$ if we have (\ref{eqn:condition_on_p_multi_intr}) and 
	\[
		\limsup_{\boldsymbol{\theta} \rightarrow \boldsymbol{\theta}_0} \frac{\sum_{\boldsymbol{\xi}  \in \Omega(\boldsymbol{\theta})\setminus \{\boldsymbol{\theta}\} }\left|{p}(\boldsymbol{\xi})\right|^2}	
		{\sum_{\zeta\in\{x,y\}}\mathbf{f}_{B_{\zeta}}(\boldsymbol{\theta})\mathbf{g}_{\zeta}(\boldsymbol{\theta})\mathbf{f}_{B_{\zeta}}^H(\boldsymbol{\theta})}<\infty.
	\]
	Since the enumerator is a scalar, this second condition is the finiteness of the following limit superior:
	\[
		\limsup_{\boldsymbol{\theta} \rightarrow \boldsymbol{\theta}_0} \frac{1}
		{\sum_{\zeta\in\{x,y\}}\left(\mathbf{f}_{B_{\zeta}}(\boldsymbol{\theta})/\sum_{\boldsymbol{\xi}  \in \Omega(\boldsymbol{\theta})\setminus \{\boldsymbol{\theta}\} }\left|{p}(\boldsymbol{\xi})\right|\right)\mathbf{g}_{\zeta}(\boldsymbol{\theta})\left(\mathbf{f}_{B_{\zeta}}^H(\boldsymbol{\theta})/\sum_{\boldsymbol{\xi}  \in \Omega(\boldsymbol{\theta})\setminus \{\boldsymbol{\theta}\} }\left|{p}(\boldsymbol{\xi})\right|\right)}.
	\]
	Recalling again that the matrices $\mathbf{g}_{\zeta}(\boldsymbol{\theta})$ are HPD for all $\boldsymbol{\theta}$, the latter condition corresponds to requiring that
	\begin{equation}\label{eq:limsup_Czero_sc1}
		\limsup_{\boldsymbol{\theta} \rightarrow \boldsymbol{\theta}_0} \frac{\mathbf{f}_{B_x}(\boldsymbol{\theta})}{\sum_{\boldsymbol{\xi}  \in \Omega(\boldsymbol{\theta})\setminus \{\boldsymbol{\theta}\} }\left|{p}(\boldsymbol{\xi})\right|} \neq \mathbf{0} \qquad \mbox{or} \qquad
		\limsup_{\boldsymbol{\theta} \rightarrow \boldsymbol{\theta}_0} \frac{\mathbf{f}_{B_y}(\boldsymbol{\theta})}{\sum_{\boldsymbol{\xi}  \in \Omega(\boldsymbol{\theta})\setminus \{\boldsymbol{\theta}\} }\left|{p}(\boldsymbol{\xi})\right|} \neq \mathbf{0}.
	\end{equation}
\end{Remark}

\begin{Remark}
 {The  presented analysis has been performed considering circulant matrices to simplify calculations and proofs. However, we consider in next sections  applications which usually produce Toeplitz/Toeplitz-like structures. Nevertheless, the convergence result in Theorem \ref{thm:Notay_symbol_2D_block} can be extended and it} holds also for multilevel block Toeplitz matrices generated by trigonometric polynomials. Indeed, they  are a low rank correction of multilevel block circulant matrices with the same generating function. The convergence could slightly deteriorate, but the combination of the proposed multigrid method with Krylov methods kills the outliers and guarantees the efficiency proven in the circulant case.
\end{Remark}

\section{Saddle-point matrices stemming from the Stokes equation}\label{sec:stokes}
In this section we are interested in applying the multigrid method described before to large linear systems stemming from the finite element approximation of the Stokes equation. 
The Stokes equation has the form
\begin{equation}\label{eq:thegeneralproblem}
\left\lbrace\begin{array}{rll}
\Delta \mathbf{u} + \mathop{\mathrm{grad}} p & = -\Phi \quad & \text{in}\, \Omega,\\
\mathop{\mathrm{div}} \mathbf{u} & = 0 \quad & \text{in}\, \Omega,\\
\mathbf{u} & = \mathbf{g}_D \quad & \text{on}\, \Gamma_D,\\
\frac{\partial \mathbf{u}}{\partial \nu} & = \mathbf{g}_N \quad & \text{on}\, \Gamma_N,
\end{array}\right.
\end{equation}
where $\mathbf{u}$ is the velocity and $p$ is the pressure, by $\Gamma_D$ we denote the Dirichlet boundary part of $\Gamma := \partial\Omega$ and by $\Gamma_N$ the Neumann part. For the construction and analysis of the system matrices we consider the domain $\Omega = (0,1)^2$, homogenous Dirichlet boundary conditions are attained for $x_1 = 0$ and $x_2 = 1$, homogenous Neumann boundary conditions for $x_1 = 1$ and $x_2 = 0$.

\subsection{A TGM for a Finite Element discretization of the Stokes problem}
\label{ssec:Q1_iso_Q1_matthias}
The weak form of \eqref{eq:thegeneralproblem} is given by
\begin{equation*}
\begin{array}{rll}
a(\mathbf{u},\mathbf{v}) - b(\mathbf{v},q) & = (\mathbf{\Phi},\mathbf{v}), \quad & \text{for all $\mathbf{v}$},\\
b(\mathbf{u},q) & = 0, \quad & \text{for all $\mathbf{q}$},
\end{array}
\end{equation*}
where
\begin{equation*}
a(\mathbf{u},\mathbf{v}) = \int\limits_\Omega \nabla \mathbf{u} \cdot \nabla \mathbf{v}\ dx \quad \text{and} \quad b(\mathbf{u},q) = \int\limits_\Omega \mathop{\mathrm{div}} \mathbf{u}\ q\ dx.
\end{equation*}
Higher order elements are needed for the discretization of the velocities $\mathbf{u}$ than for the discretization of the pressure $p$, e.g., Q2-Q1 Taylor-Hood-elements. Alternatively, linear elements can be used for the {velocities} but using a subdivision of the element, corresponding to a 2-times finer---or once refined---mesh for the {velocities}. We choose the latter approach that is often called Q1-iso-Q2/Q1 \cite[ {Section VI.6}]{MR1115205}. The properties of the system matrix are similar to those of the system matrix produced by Q2-Q1 elements. The advantage of using Q1-iso-Q2/Q1 is that the discrete operator representing $a$ is more sparse than that for Q2-Q1 and thus cheaper to apply. {We consider uniform discretization with same numbers of points in both directions $x$ and $y$ equals to $n$. } The resulting algebraic system of equations is given by
\begin{equation}
\label{eq:syst_q1_iso_q2}
	{{A}}_N\mathbf{x} = \left[\begin{array}{cc|c}
	{A_x} & O & B_x \\
	O & A_y& B_y \\[-0.1em] \hline
	B_x^H & B_y^H & O
	\end{array}\right]\begin{bmatrix}
	 \mathbf{u}\\
	  \mathbf{p}  
\end{bmatrix}	  = \begin{bmatrix}
	 \mathbf{b_u}\\
	  \mathbf{b_p}  
\end{bmatrix}=	 \mathbf{b},
\end{equation}
where $A_x,A_y \in \mathbb{R}^{4n^2\times 4n^2}$ and $B_x, B_y \in \mathbb{R}^{4n^2\times n^2}$, so $N = 9n^2$.
In particular, the matrices $A_x$ and $A_y$ are non-negative and they are a permutation of $4\times 4$ block bi-level Toeplitz matrices with partial dimension $n$.
Precisely, let $\mathbf{e}_i$  be the $i$th column of the identity matrix of size $4n^2$, we can define a proper $4n^2\times 4n^2$ permutation matrix,  $\Pi_4=[P_1|P_2|\dots|P_4]$, $P_l\in \mathbb{R}^{4n^2\times n^2},\, l=1,\dots,4$,  such that the $k$th column of $P_l,\, l=1,\dots,4$, is $e_{l+4(k-1)}$. The matrix $\Pi_4$ transforms the matrix ${A_x}$ into the following a $4\times 4$ block bi-level Toeplitz matrix
\[ T_{\mathbf{n}}(\mathbf{f}_{A_x}(\theta_1,\theta_2))=\Pi_4{A_x}\Pi_4^T,\]
where 
\begin{equation}\label{eq:fA}
\mathbf{f}_{A_x}(\theta_1,\theta_2)=
-\frac{1}{3}\begin{bmatrix}
-8 & 1+{\rm e}^{-\iota\theta_1} & 1+{\rm e}^{-\iota\theta_2} &\varphi(-\theta_1,-\theta_2)\\
1+{\rm e}^{\iota\theta_1} & -8 & \varphi(\theta_1,-\theta_2) & 1+{\rm e}^{-\iota\theta_2}\\
1+{\rm e}^{\iota\theta_2}& \varphi(-\theta_1,\theta_2) & -8 & 1+{\rm e}^{-\iota\theta_1}\\
\varphi(\theta_1,\theta_2) &  1+{\rm e}^{\iota\theta_2} &  1+{\rm e}^{\iota\theta_1} & -8
\end{bmatrix},
\end{equation}
with $\varphi(\theta_1,\theta_2)=1+{\rm e}^{\iota\theta_1}+{\rm e}^{\iota\theta_2}+{\rm e}^{\iota(\theta_1+\theta_2)}$. An analogous transformation can be applied to ${A_y}$ and the generating function is $\mathbf{f}_{A_y}(\theta_1,\theta_2)=\mathbf{f}_{A_x}(\theta_2,\theta_1)$.

We consider an analogous $N\times N$ permutation matrix,  $\Pi_9=[P_1|P_2|\dots|P_9]$, $P_l\in \mathbb{R}^{N\times n^2},\, l=1,\dots,9$,  such that the $k$th column of $P_l,\, l=1,\dots,9$, is $e_{l+9(k-1)}$, with $e_i$ being the $i$th column of the identity matrix of size $9n^2$. We can transform the global matrix $\mathcal{A}_N$ in a $9\times 9$ block bi-level system. This property will be used later when we will exploit the information of the permuted matrix
$
	G_N=\Pi_9 \mathcal{A}_N \Pi_9^T=T_{\mathbf{n}}(\mathbf{f}),
$
	where
	$T_{\mathbf{n}}(\mathbf{f})$ is the bi-level $9\times 9$ block Toeplitz
		$
		T_{\bf{n}}(\mathbf{f}) =\left[\hat{\mathbf{f}}_{\mathbf{i}-\mathbf{j}}\right]_{\mathbf{i},\mathbf{j}=\mathbf{1}}^{\bf n}\in\mathbb{C}^{N\times N}
		$
		generated by $\mathbf{f}:[-\pi,\pi]^2\rightarrow \mathbb{C}^{9\times 9}$. In particular,
		
	\begin{equation}\label{eq:symbol_f}
	\mathbf{f}(\theta_1,\theta_2)=\begin{bmatrix}
	\mathbf{f}_{A_x}(\theta_1,\theta_2) & O & \mathbf{f}_{B_x}(\theta_1,\theta_2)\\
	O& \mathbf{f}_{A_y}(\theta_1,\theta_2)& \mathbf{f}_{B_y}(\theta_1,\theta_2)\\
	\mathbf{f}_{B_x}^H(\theta_1,\theta_2)& \mathbf{f}_{B_y}^H(\theta_1,\theta2)& 0
	\end{bmatrix},
\end{equation}		
where
\begin{equation}
\label{eq:simbolo_global}
	\mathbf{f}_{B_x}(\theta_1,\theta_2)=\begin{bmatrix}
	\frac{1}{48}(-\iota\sin\theta_1)(1+2\cos\theta_2)\\
	\frac{1}{24}(1-{\rm e}^{\iota\theta_1})(5+\cos\theta_2)\\
	\frac{1}{8}(\iota\sin\theta_1)(1+{\rm e}^{\iota\theta_2})\\
	\frac{1}{8}(1-{\rm e}^{\iota\theta_1})(1+{\rm e}^{\iota\theta_2})
	\end{bmatrix}\!\!,\,
	\mathbf{f}_{B_y}(\theta_1,\theta_2)=\begin{bmatrix}
	\frac{1}{48}(1+2\cos\theta_1)(-\iota\sin\theta_2)\\
		\frac{1}{8}(1+{\rm e}^{\iota\theta_1})(-\iota\sin\theta_2)\\
	\frac{1}{24}(5+\cos\theta_1)(1-{\rm e}^{\iota\theta_2})\\
	\frac{1}{8}(1+{\rm e}^{\iota\theta_1})(1-{\rm e}^{\iota\theta_2})
	\end{bmatrix}\!\!.
\end{equation}	

{To solve efficiently the system (\ref{eq:syst_q1_iso_q2}) with TGM satisfying the assumptions of Theorem 	\ref{thm:Notay_symbol_2D_block}},  we investigate the structure of the coefficient matrix  ${\mathcal{A}}_N $. It is a block matrix of the form
\begin{equation*}
{\mathcal{A}}_N = \left[\begin{array}{ccc}
	\tilde{A}& B^T \\
	B & -C
	\end{array}\right],
\end{equation*}
with  $C=O$, $B=[B_x,B_y]$ and $\tilde{A}=\begin{bmatrix}
{A_x} & O\\
	O & {A_y}
\end{bmatrix}$. 

\subsection{Spectral analysis of $\tilde{A}$ and choice of $P_{\tilde{A}}$ }
\label{ssec:A}

{
Lemma \ref{lem:alpha_multi_block} and Theorem \ref{thm:Notay_symbol_2D_block} highlight the importance of the block $\tilde{A}$  for the choice of the parameter $\alpha$ and the construction of an efficient operator $P_{\tilde{A}}$. 
In practice we can just consider the matrix  $A_x$ and the related generating function $\mathbf{f}_{A_x}$, since $A_y$ and $\mathbf{f}_{A_y}$ enjoy the same spectral properties. For this reason in the following we will avoid the specification $x$ or $y$.}
We already know that $A$ is similar to a bi-level $4\times 4$ block Toeplitz $T_{\bf{n}}(\mathbf{f}_A)$. The following proposition provides us with important spectral information on $\mathbf{f}_A$, and consequently $T_{\bf{n}}(\mathbf{f}_A)$.
\begin{Proposition}\label{prop:fA}
Let $\mathbf{f}_A$ be the $4\times 4$ matrix-valued function defined in (\ref{eq:fA}) and $T_{\bf{n}}(\mathbf{f}_A)=\Pi_4A\Pi_4^T$ the associated Toeplitz matrix. Then the four eigenvalue functions of $\mathbf{f}_A$  are 
 \begin{equation}
 \label{eq:eig_fa_explicit}
 \begin{split}
 & \lambda_1(\mathbf{f}_A)= 3-\frac{1}{3}{\rm e}^{-\iota \left(\frac{\theta_1+\theta_2}{2}\right)}\left( {\rm e}^{\iota \theta_1}+{\rm e}^{\iota \frac{\theta_1}{2}}+1\right)
\left( {\rm e}^{\iota \theta_2}+{\rm e}^{\iota \frac{\theta_2}{2}}+1\right);\\
& \lambda_2(\mathbf{f}_A)=3-\frac{1}{3}{\rm e}^{-\iota \left(\frac{\theta_1+\theta_2}{2}\right)}\left( {\rm e}^{\iota \theta_1}-{\rm e}^{\iota \frac{\theta_1}{2}}+1\right)
\left( {\rm e}^{\iota \theta_2}-{\rm e}^{\iota \frac{\theta_2}{2}}+1\right);\\
& \lambda_3(\mathbf{f}_A)=3+\frac{1}{3}{\rm e}^{-\iota \left(\frac{\theta_1+\theta_2}{2}\right)}\left( {\rm e}^{\iota \theta_2}-{\rm e}^{\iota \frac{\theta_2}{2}}+1\right)
\left( {\rm e}^{\iota \theta_1}+{\rm e}^{\iota \frac{\theta_1}{2}}+1\right);\\
&\lambda_4(\mathbf{f}_A)=3+\frac{1}{3}{\rm e}^{-\iota \left(\frac{\theta_1+\theta_2}{2}\right)}\left( {\rm e}^{\iota \theta_1}-{\rm e}^{\iota \frac{\theta_1}{2}}+1\right)
\left( {\rm e}^{\iota \theta_2}+{\rm e}^{\iota \frac{\theta_2}{2}}+1\right).\\
 \end{split}
 \end{equation}
 Moreover, 
 \begin{enumerate}
 \item the minimum eigenvalue function of $ \lambda_1(\mathbf{f}_A)$ is non-negative with a zero of order 2 in $\boldsymbol{\theta}_0=(0, 0)$. Furthermore, $e_4$ is the eigenvector of $\mathbf{f}_A(\boldsymbol{\theta}_0)$ associated with $\lambda_1(\mathbf{f}_A(\boldsymbol{\theta}_0))=0$, where ${\rm \textbf{e}}_{4}=[1,1,1,1]^T$. 
 \item the minimal eigenvalue of $T_{\bf{n}}(\mathbf{f}_A)$ (and consequently of $A$) goes to zero as $(4n^2)^{-1}$;
 \item the condition number $\kappa(T_{\bf{n}}(\mathbf{f}_A))$ of $T_{\bf{n}}(\mathbf{f}_A)$ (and consequently of $A$) is proportional to  $O(n^2)$.
 \end{enumerate}
  
\end{Proposition}
\begin{proof}
The function $\mathbf{f}_A$ can be written in a more compact form as
\begin{equation}
\label{eq:symplification_fA}
\mathbf{f}_A(\theta_1,\theta_2)=\frac{9}{3} I_4-\frac{1}{3} h(\theta_2)\otimes h(\theta_1)=3I_4-\frac{1}{3} h(\theta_2)\otimes h(\theta_1),
\end{equation}
where $h(\theta)=\begin{bmatrix}
1 & {\rm e}^{-\iota \theta}+1\\
{\rm e}^{\iota \theta}+1 & 1
\end{bmatrix}.
$
This implies that $T_{\bf{n}}(\mathbf{f}_A)$ is the SPD matrix 
\begin{equation}
\label{eq:perm_A}
T_{\bf{n}}(\mathbf{f}_A)=3I_{4n^2}-\frac{1}{3} T_{{n}}(h(\theta_2))\otimes T_{{n}}(h(\theta_1)).
\end{equation}
Moreover, formula (\ref{eq:symplification_fA})  and the properties of the tensor product imply that the four eigenvalue functions $\lambda_k(\mathbf{f}_A), \, k=1,\dots,4$ of $\mathbf{f}_A$ are given by the following combination of the two eigenvalue functions of $h({\theta})$ $$\left\{3-\frac{1}{3} \lambda_{i}(h(\theta_2))\lambda_{j}(h(\theta_1)\right\}_{i,j=1}^2.$$ 
The eigenvalue functions of $h({\theta})$ can be computed analytically and they are given by
 \[\lambda_1(h(\theta))= -{\rm e}^{-\iota \frac{\theta}{2}}({\rm e}^{\iota {\theta}}-{\rm e}^{\iota \frac{\theta}{2}} + 1); \quad \lambda_2(h(\theta))={\rm e}^{-\iota \frac{\theta}{2}}({\rm e}^{\iota {\theta}}+{\rm e}^{\iota \frac{\theta}{2}} + 1).\] 
 Then, using their combination, we derive the expressions of the four eigenvalue functions of $\mathbf{f}_A(\theta_1,\theta_2)$ which are exactly those given in formula (\ref{eq:eig_fa_explicit}).
Moreover, by direct computation, we have
\[\mathbf{f}_A(0,0){\rm \textbf{e}}_{4}=\begin{bmatrix}
 &   8/3    &       -2/3 &          -2/3   &        -4/3     \\
  &    -2/3  &          8/3&           -4/3 &          -2/3    \\ 
   &   -2/3  &         -4/3 &           8/3  &         -2/3     \\
    &  -4/3   &        -2/3   &        -2/3   &         8/3 \\
\end{bmatrix}{\rm \textbf{e}}_{4}=0{\rm \textbf{e}}_{4}.
 \]
 Then, ${\rm \textbf{e}}_{4}$  is an eigenvector of $\mathbf{f}_A(\boldsymbol{\theta}_0)$ associated with $0$.
In addition, the minimal eigenvalue function is
\[\lambda_1(\mathbf{f}_A)= 3-\frac{1}{3}{\rm e}^{-\iota \left(\frac{\theta_1+\theta_2}{2}\right)}\left( {\rm e}^{\iota \theta_1}+{\rm e}^{\iota \frac{\theta_1}{2}}+1\right)
\left( {\rm e}^{\iota \theta_2}+{\rm e}^{\iota \frac{\theta_2}{2}}+1\right),\]
and it is straightforward to see that is a non-negative function over $[-\pi,\pi]$ and it is such that 
 \begin{align*}
\lambda_1(\mathbf{f}_A){|_{(0,0)}}&= 0,\\
\frac{\partial \, \lambda_1(\mathbf{f}_A)(\theta_1, \theta_2)}{\partial{\theta_1}}{\left|_{(0,0)}\right.}&=\frac{\partial \, \lambda_1(\mathbf{f}_A)(\theta_1, \theta_2)}{\partial{\theta_2}}{\left|_{(0,0)}\right.}=0,
\end{align*}
\begin{align*}
\frac{\partial^2 \, \lambda_1(\mathbf{f}_A)(\theta_1, \theta_2)}{\partial{\theta_2}\partial{\theta_1}}{\left|_{(0,0)}\right.}&= \frac{\partial^2 \, \lambda_1(\mathbf{f}_A)(\theta_1, \theta_2)}{\partial{\theta_1}\partial{\theta_2}}{\left|_{(0,0)}\right.}=0, 
\end{align*}
\begin{align*}
\frac{\partial^2 \, \lambda_1(\mathbf{f}_A)(\theta_1, \theta_2)}{\partial{\theta_1^2}}{\left|_{(0,0)}\right.}&=\frac{\partial^2 \, \lambda_1(\mathbf{f}_A)(\theta_1, \theta_2)}{\partial{\theta_2^2}}{\left|_{(0,0)}\right.}= \frac{1}{2}.\\
\end{align*} Therefore, $\lambda_1(\mathbf{f}_A)$ has a zero of order 2 in $\boldsymbol{\theta}_0=(0, 0)$.
In light of the third item of Theorem \ref{thm:loc-extr-s}, we conclude that the minimal eigenvalue of $T_{\bf{n}}(\mathbf{f}_A)$  goes to zero as $(4n^2)^{-1}$ and, by similitude, that of $A$ as well.  Furthermore,
\begin{equation}
\label{eq:max_eig_fA}
M_s=\max_{(\theta_1,\theta_2)\in [-\pi,\pi]}\lambda_4(\mathbf{f}_A)=\max_{(\theta_1,\theta_2)\in [-\pi,\pi]}\lambda_3(\mathbf{f}_A)=4,
\end{equation}
 and from Theorem \ref{thm:loc-extr-s} we have that the spectrum of $T_{\bf{n}}(\mathbf{f}_A)$ is contained in $(0,4).$ Consequently,
the condition number $\kappa(T_{\bf{n}}(\mathbf{f}_A))$ of $T_{\bf{n}}(\mathbf{f}_A)$ is 
$$\kappa(T_{\bf{n}}(\mathbf{f}_A))=\frac{\lambda_{\max}(T_{\bf{n}}(\mathbf{f}_A))}{\lambda_{\min}(T_{\bf{n}}(\mathbf{f}_A))}\le \frac{4}{\lambda_{\min}(T_{\bf{n}}(\mathbf{f}_A))}\thickapprox n^2. $$
The similitude between $A$ and  $T_{\bf{n}}(\mathbf{f}_A)$ concludes the proof of the Proposition.
 	\end{proof}
A graphical confirmation of the result of the Proposition \ref{prop:fA} concerning the behavior of the eigenvalue functions is given in Figure \ref{fig:eig_fA} where the plot of $\lambda_i(\mathbf{f}_A)$, $i=1,\dots,4,$ is shown.

\begin{figure}[htb]
\centering
\includegraphics[width=0.39\textwidth]{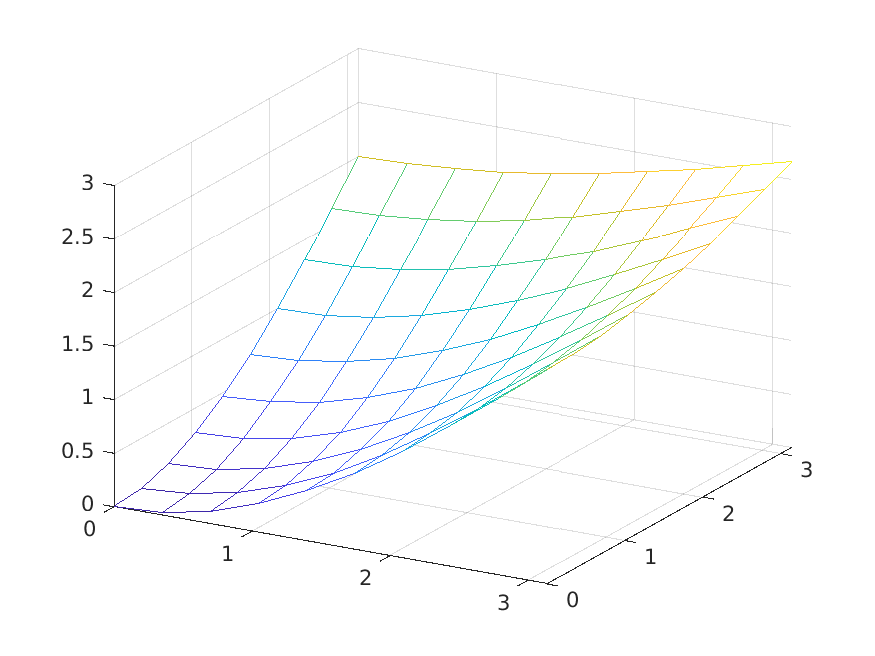}
\label{valutazioni_lambda1_n=40_fA}
\includegraphics[width=0.39\textwidth]{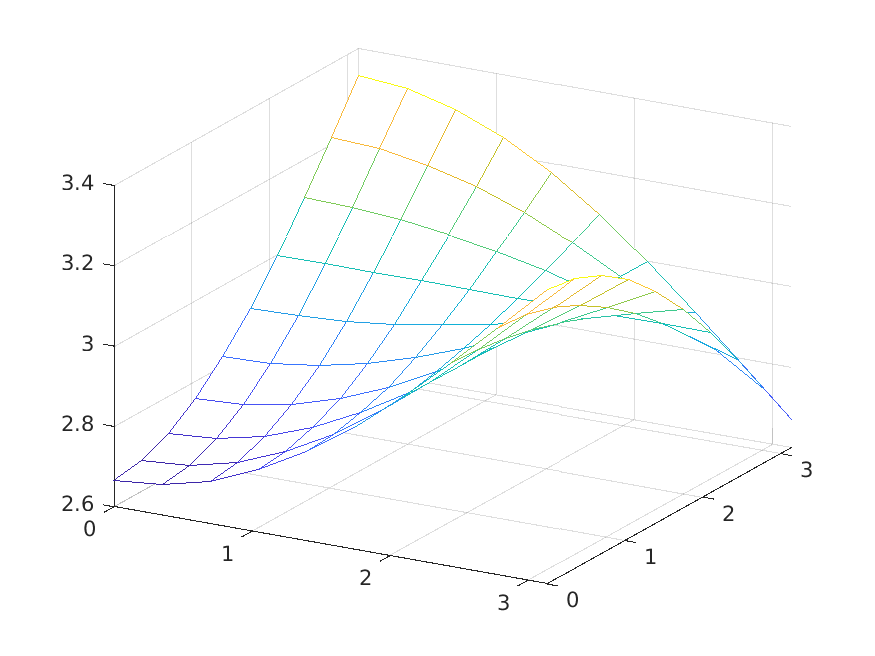}
\label{valutazioni_lambda2_n=40_fA}
\includegraphics[width=0.39\textwidth]{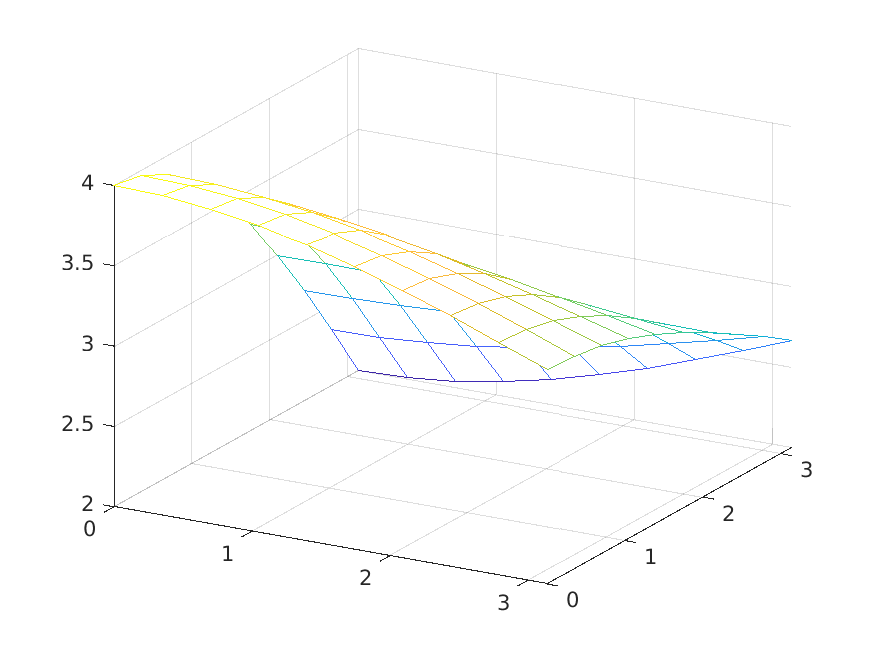}
\label{valutazioni_lambda3_n=40_fA}
\includegraphics[width=0.39\textwidth]{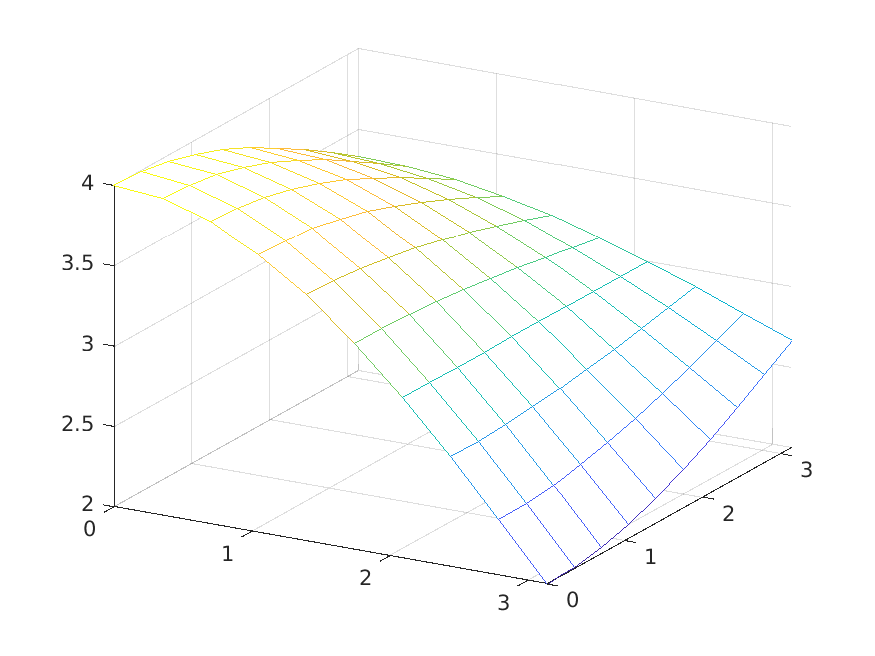}
\label{valutazioni_lambda4n=40_fA}
\caption{Plot of the eigenvalues functions $\lambda_l(\mathbf{f}_A)$, $l=1,\ldots,4$ when $n=10$.}\label{fig:eig_fA}
\end{figure}

{
\begin{Remark}
Because of the spectral properties of the symbol $\mathbf{f}_A$ and $T_{\bf{n}}(\mathbf{f}_A)$, we need to choose  $$\alpha< 2\max_{j=1,\dots,4}\left\{\left(\hat{a}_0^{(j,j)}(\mathbf{f}_{A})\right)^{-1}\|\lambda_j\left(\mathbf{f}_{A}\right)\|_\infty\right\}^{-1}=2\left(\,\frac{3}{8}\,4\right)^{-1}=\frac{4}{3}$$
\end{Remark}
 {In the following we will fix $\alpha$ equal to $\frac{2}{3}$ which is the middle value in the range of the admissible values $(0, 4/3)$. This choice is motivated by the analysis performed in \cite[Section 5]{MR3439215} where it is  investigated the sensitivity with respect to variations around the value $\alpha\approx \left(\lambda_{\rm max}\left(D_A^{-1} A\right)\right)^{-1} $ - which corresponds to the middle point of the admissible values.}


{As suggested by Theorem \ref{thm:Notay_symbol_2D_block}, we need to construct 
\begin{equation}
\label{eq:P_Atilde}
{P}_{\tilde{A}}=\begin{bmatrix}
  \mathcal{C}_{\mathbf{n}}(\mathbf{p}_{A_x})(K_{\mathbf{n}}^T\otimes I_{s_a}) & \\
   &   \mathcal{C}_{\mathbf{n}}(\mathbf{p}_{A_y})(K_{\mathbf{n}}^T \otimes I_{s_a})
  \end{bmatrix}, \quad P_{\hat{C}}=\mathcal{C}_\mathbf{n}(\mathbf{p}_{\hat{C}})(K_\mathbf{n}^T\otimes I_{s_c})
\end{equation}
  
   with $\mathbf{p}_{A_x}$, $\mathbf{p}_{A_y}$, and $\mathbf{p}_{\hat{C}}$ proper bi-variate matrix-valued trigonometric polynomials. Since $\mathbf{f}_{A_x}$ and  $\mathbf{f}_{A_y}$ have the same spectral properties, we fix $\mathbf{p}_{A_x}=\mathbf{p}_{A_y}=: \mathbf{p}_{A}.$
In the following we discuss the choice of the  polynomial $\mathbf{p}_{A}$.
}

We recall that a classical approach is that of constructing $P_{{A}}$ as the multilevel linear interpolation projector. That is,
\begin{equation}\label{eq:P_A}
{P}_{{A}}=\Pi_4^T \left(T_{\mathbf{n}}(\mathbf{p}^{(4)})\left(K_{\mathbf{n,k}}\otimes I_{\textcolor{black}{4}}\right)\right)\Pi_4,
\end{equation}
where $K_{\mathbf{n,k}}=K_{{n,k}}\otimes K_{{n,k}}$, $k=\frac{n-1}{2}$, $n$ odd, and 

\[K_{n,k} = \left[\begin{array}{cccccccc}
		0 &  & & &\\
		1 &  & & &\\
        0 & 0& & &\\
	\vdots& 1& & &\\	
		  & 0& & &\\
		  & \vdots& & &\\
		  &  & & &\vdots\\
		  &  & & &1\\
	      &  & & &0\\		
		\end{array}\right]_{n\times k}.\]

 Moreover, we choose as $\mathbf{p}^{(4)}(\theta_1,\theta_2)=\mathbf{p}^{(2)}(\theta_1) \otimes \mathbf{p}^{(2)}(\theta_2)$, that is the Kronecker product of the univariate trigonometric polynomial 
 \begin{equation}\label{eq:p_L2}
 \mathbf{p}^{(2)}(\theta)=\begin{bmatrix}
1 +{\rm e}^{-\iota \theta}& {\rm e}^{\iota \theta}+1\\
2{\rm e}^{-\iota \theta} & 2
\end{bmatrix}.
 \end{equation}
Next theorem highlights  the properties of the polynomial $\mathbf{p}^{(2)}(\theta)$ and its multivariate version $\mathbf{p}^{(4)}(\theta_1, \theta_2)$ in order to show that the grid transfer operator ${P}_{{A}}$ fulfills  the conditions for the convergence and optimality of the TGM methods proposed in \cite[Section 6]{block_multigrid2}, here conditions (\ref{eqn:condition_on_p_multi_intr})
- (\ref{eqn:condition_on_s_f_multi_intr}).

\begin{Theorem}\label{thm:tgm_conv_Pa}
Let $\mathbf{p}^{(2)}(\theta)$ be the $2\times 2$ trigonometric polynomial defined in (\ref{eq:p_L2}) and $\mathbf{p}^{(4)}(\theta_1,\theta_2)=\mathbf{p}^{(2)}(\theta_1) \otimes \mathbf{p}^{(2)}(\theta_2)$. Consider $ \mathbf{f}_A(\theta_1,\theta_2)$ defined in (\ref{eq:fA}).  Then, $\mathbf{p}^{(4)}(\theta_1,\theta_2)=\mathbf{p}^{(2)}(\theta_1) \otimes \mathbf{p}^{(2)}(\theta_2)$ is such that:
\begin{itemize}
\item \begin{equation}\label{eqn:condition_on_p_multi}
	\sum_{\xi \in \Omega(\boldsymbol{\theta})}\mathbf{p}^{(4)}(\xi)^{H}\mathbf{p}^{(4)}(\xi)>0, \quad\forall\, \boldsymbol{\theta} \in[0,2\pi)^{ {2}},
\end{equation}
 where $\Omega(\theta)=\left\{({\theta_1}, \theta_2) , ({\theta_1+\pi}, \theta_2), ({\theta_1}, \theta_2+\pi), ({\theta_1}+\pi, \theta_2+\pi) \right\}$.
 \item The function
$
\mathbf{s}^{(4)}(\boldsymbol{\theta}) = \mathbf{p}^{(4)}(\boldsymbol{\theta})\left(\sum_{\xi \in \Omega(\boldsymbol{\theta}) }\mathbf{p}^{(4)}(\xi)^{H}\mathbf{p}^{(4)}(\xi)\right)^{-1}\mathbf{p}^{(4)}(\boldsymbol{\theta})^{H}$
is well-defined for all $\boldsymbol{\theta} \in[0,2\pi)^{ {2}}$ and \begin{equation}\label{eqn:condition_on_s_multi}
	\mathbf{s}^{(4)}(\boldsymbol{\theta} _0){\rm \textbf{e}}_{4}={\rm \textbf{e}}_{4},
\end{equation}
where ${\rm \textbf{e}}_{4}=[1,1,1,1]^T$ and $\boldsymbol{\theta}_0=(0,0) $.
\item \begin{equation}\label{eqn:condition_on_s_f_multi}
\lim_{\boldsymbol{\theta} \rightarrow \boldsymbol{\theta}_0} (\lambda_1(\mathbf{f}_A(\boldsymbol{\theta})))^{-1}(1-\lambda_{1}(\mathbf{s}^{(4)}(\boldsymbol{\theta})))=c,
\end{equation}
 where $c\in\mathbb{R}$ is a constant and $\lambda_1(\mathbf{f}_A)$ and $\lambda_{1}(\mathbf{s}^{(4)})$ are the minimal eigenvalue functions of $\mathbf{f}_A$ and $\mathbf{s}^{(4)}$, respectively. 
\end{itemize}
\end{Theorem}

\begin{proof}
Given ${\rm \textbf{e}}_{2}=[1,1]^T$, the trigonometric polynomial $\mathbf{p}^{(2)}$  by direct computation fulfils the following:
 \begin{enumerate}
 \item [1)]  $\mathbf{p}^{(2)}(0)\,{\rm e}_{2}=4\, {\rm e}_{2}$.
 \item  [2)]  $\mathbf{p}^{(2)}(\pi)\,{\rm e}_{2}=0\, {\rm e}_{2} $.
   \item  [3)]  $\mathbf{p}^{(2)}(0)^H\,{\rm e}_{2}=4\, {\rm e}_{2}$. 
   \item [4)] $\mathbf{p}^{(2)}(\theta)^{H}\mathbf{p}^{(2)}(\theta)+\mathbf{p}^{(2)}(\theta+\pi)^{H}\mathbf{p}^{(2)}(\theta+\pi)=\begin{bmatrix}
    12 &2{\rm e}^{2\iota \theta} + 2\\
2{\rm e}^{-2\iota \theta} + 2&                 12   
\end{bmatrix}$.
 \end{enumerate}
 Item $4)$ clearly implies that $\mathbf{p}^{(2)}(\theta)^{H}\mathbf{p}^{(2)}(\theta)+\mathbf{p}^{(2)}(\theta+\pi)^{H}\mathbf{p}^{(2)}(\theta+\pi)$ is a positive definite matrix, $\forall \theta \in [0,2\pi)$. The items $1)-3)$, together with  \cite[Lemma 4.3]{block_multigrid2},  imply that
 \[\mathbf{s}^{(2)}(0){\rm \textbf{e}}_{2} = {\rm \textbf{e}}_{2}, \] 
 where $\mathbf{s}^{(2)}=\mathbf{p}^{(2)}(\theta)\left(\mathbf{p}^{(2)}(\theta)^{H}\mathbf{p}^{(2)}(\theta)+\mathbf{p}^{(2)}(\theta+\pi)^{H}\mathbf{p}^{(2)}(\theta+\pi)\right)^{-1}\mathbf{p}^{(2)}(\theta)^{H}$, is well-defined for all $\theta\in [0,2\pi)$.
 Because of the tensor structure of $\mathbf{ p}^{(4)}$, the result in \cite[Lemma 6.2]{block_multigrid2} ensures that it verifies the analogous positivity condition given in  (\ref{eqn:condition_on_p_multi}).   Consequently, the quantity \[
\mathbf{s}^{(4)}(\boldsymbol{\theta}) = \mathbf{p}^{(4)}(\boldsymbol{\theta})\left(\sum_{\xi \in \Omega(\boldsymbol{\theta}) }\mathbf{p}^{(4)}(\xi)^{H}\mathbf{p}^{(4)}(\xi)\right)^{-1}\mathbf{p}^{(4)}(\boldsymbol{\theta})^{H}\]
is well-defined for all $\boldsymbol{\theta}=(\theta_1,\theta_2) \in[0,2\pi)^{ {2}}$. In addition, by \cite[Lemma 6.2]{block_multigrid2}, we have $\mathbf{s}^{(4)}(\theta_1,\theta_2)=\mathbf{s}^{(2)}(\theta_1)\otimes \mathbf{s}^{(2)}(\theta_2)$.
Writing ${\rm \textbf{e}}_{4}$ as ${\rm \textbf{e}}_{2}\otimes {\rm \textbf{e}}_{2}$ and exploiting the properties of the Kronecker product, we have
\begin{equation*}
\begin{split}
&\mathbf{s}^{(4)}(\boldsymbol{\theta}_0){\rm \textbf{e}}_{4}=(\mathbf{s}^{(2)}(0)\otimes \mathbf{s}^{(2)}(0))({\rm \textbf{e}}_{2}\otimes {\rm \textbf{e}}_{2})=\\
&(\mathbf{s}^{(2)}(0){\rm \textbf{e}}_{2})\otimes  (\mathbf{s}^{(2)}(0){\rm \textbf{e}}_{2})={\rm \textbf{e}}_{2}\otimes {\rm \textbf{e}}_{2}={\rm \textbf{e}}_{4}.
\end{split}
\end{equation*}
which concludes the proof of (\ref{eqn:condition_on_s_multi}). The last thing to be proved is that
\begin{equation}
\limsup_{\boldsymbol{\theta} \rightarrow \boldsymbol{\theta}_0} \lambda_{1}(\mathbf{f}_A(\boldsymbol{\theta}))^{-1}(1-\lambda_{1}(\mathbf{s}^{(4)}(\boldsymbol{\theta})))<+\infty.
\end{equation}
Note that $\mathbf{s}^{(2)}$ is an algebraic projector, that is, it can be verified by direct computation that $\mathbf{s}^2(\theta)-\mathbf{s}(\theta)=\mathbf{0}$. Since $\lambda_{1}(\mathbf{s}(0))= 1$ and the eigenvalue functions are continuous, we have $\lambda_1(\mathbf{s}^{(2)}(\theta))\equiv 1$ and so $\lambda_1(\mathbf{s}^{(4)})$ is identically 1 as well. Moreover, by (\ref{eq:eig_fa_explicit}), the function $\lambda_1(\mathbf{f}_A)$ is different from $0$ in a neighbourhood of $\boldsymbol{\theta}_0$. Then, the bound in (\ref{eqn:condition_on_s_f_multi}) holds.
\end{proof}

\subsection{Analysis of $\hat{C}$ and choice of $P_{\hat{C}}$}
\label{ssec:Chat}
In order to compute the sequence of grid transfer operators  $P_{\hat{C}}$, {we study the structure of $\hat{C}$, which in the presented setting has a (bi-level) scalar  nature}. 

Then, the associated spectral symbol will be a (bi-variate) scalar-valued function instead of matrix-valued. Multigrid methods for scalar non-negative Toeplitz matrix sequences have been deeply investigated in \cite{AD,ADS,FS1,FS2,Serra_Possio}, so we will first draft the idea on which our choice of ${P}_{\hat{C}}$ is based to see that it fits into a classical setting, referring in particular to Lemma 4.3 in \cite{Serra_Possio}. Then, we will exploit Remark \ref{rmk:results_Czero_sc1} to theoretically validate our choice.
 
Firstly, we recall that $\hat{C}$ is a non-negative matrix and 
 $\hat{C}=B(2\alpha D_{\tilde{A}}^{-1}-\alpha^2D_{\tilde{A}}^{-1}\tilde{A}D_{\tilde{A}}^{-1})B^T $.
{Because of the structure of $A$, $D_{\tilde{A}}=\frac{8}{3}I_{8n^2}$}. Moreover $\alpha=\frac{2}{3}$. Then, $\hat{C}$ is a low-rank correction of the Toeplitz matrix $T_{\mathbf{n}}({f}_{\hat{C}})$ with 
{$${f}_{\hat{C}}(\theta_1,\theta_2)=
\left(\mathbf{f}_{B_x}\left(\frac{1}{2}I_4-\frac{1}{16}\mathbf{f}_{A_{x}}\right) \mathbf{f}^H_{B_x}+\mathbf{f}_{B_y}\left(\frac{1}{2}I_4-\frac{1}{16}\mathbf{f}_{A_{y}}\right) \mathbf{f}^H_{B_y}\right)(\theta_1,\theta_2),
$$}
see Remark \ref{rmk:rectang_Cn_prod} for more details on the construction of the function ${f}_{\hat{C}}$.

The function ${f}_{\hat{C}}$ is a bivariate scalar valued function such that $\hat{a}_0({f}_{\hat{C}})=11/96$. Moreover, it is straightforward to show that  ${f}_{\hat{C}}$ has a zero of order 2 in the origin. For example we can use the procedure analogous to the proof of Proposition \ref{prop:fA} with the additional simplification that ${f}_{\hat{C}}$ is scalar valued then its eigenvalue function coincides with  ${f}_{\hat{C}}$ itself.  Figure \ref{fig:plot_fhatc} shows the plot of ${f}_{\hat{C}}$ on the grid $[0,\pi]^2$.

\begin{figure}[htb]
\centering
\includegraphics[width=0.50\textwidth]{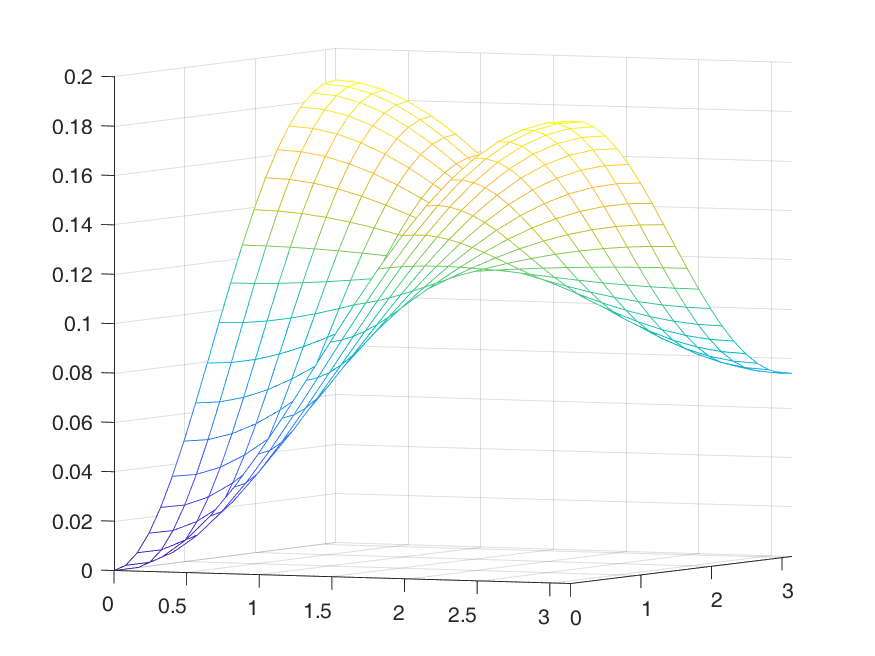}
\caption{Plot of the  function ${f}_{\hat{C}}$ on the grid $[0,\pi]^2$ with $n=20$ points.}\label{fig:plot_fhatc}
\end{figure}

Then, we construct ${P}_{{C}}$ as 
\begin{equation}\label{eq:P_C}
{P}_{{C}}=T_{\mathbf{n}}\left((2+2\cos\theta_1)(2+2\cos\theta_2)\right)K_{\mathbf{n,k}}.
\end{equation}
In order to be sure that with this choice of grid transfer operator assumption 2. of Theorem \ref{thm:Notay_symbol_2D_block} is fulfilled, we recall Remark \ref{rmk:results_Czero_sc1}. Positivity condition \eqref{eqn:condition_on_p_multi_intr} is satisfied since the sum
\begin{align*}
	&\left|(2+2\cos(\theta_1))(2+2\cos(\theta_2))\right|^2+\left|(2+2\cos(\theta_1+\pi))(2+2\cos(\theta_2))\right|^2+\\
	+&\left|(2+2\cos(\theta_1))(2+2\cos(\theta_2+\pi))\right|^2+\left|(2+2\cos(\theta_1+\pi))(2+2\cos(\theta_2+\pi))\right|^2
\end{align*}
is different from 0 for all $(\theta_1,\theta_2)\in \mathbb{R}^2$. Condition \eqref{eq:limsup_Czero_sc1} is fulfilled since the functions in \eqref{eq:simbolo_global} vanish in $(0,0)$ with a zero of order 1.

\subsection{Choice of the smoothing parameter}
\label{ssec:omega}
In order to fulfill the assumption 3.\ of Theorem \ref{thm:Notay_symbol_2D_block} in our setting we require
 that 
\begin{multline*} 
 \omega<2\min \left\{2\alpha -\alpha^2\max_j\left\{ \frac{\|\mathbf{f}_{A_x}\|_\infty}{{\hat{a}_0^{(j,j)}(\mathbf{f}_{A_x})}}, \frac{\|\mathbf{f}_{A_y}\|_\infty}{{\hat{a}_0^{(j,j)}(\mathbf{f}_{A_y})}}\right\},\right.\\
 \left.\hat{a}_0({f}_{\hat{C}})\left\|\mathbf{f}_C+\mathbf{f}_{B_x}\mathbf{f}_{A_x}^{-1}\mathbf{f}_{B_x}^H+\mathbf{f}_{B_y}\mathbf{f}_{A_y}^{-1}\mathbf{f}_{B_y}^H\right\|_\infty^{-1}\right\}.
\end{multline*} 
That is
\begin{align}
\label{eq:choice_omega_num}
 \omega&<2\min \left\{2\left(\frac{2}{3}\right) -\left(\frac{2}{3}\right)^2 \frac{4\cdot 3}{8},\right. \left.\frac{11}{96} \left\|\mathbf{f}_{B_x}\mathbf{f}_{A_x}^{-1}\mathbf{f}_{B_x}^H+\mathbf{f}_{B_y}\mathbf{f}_{A_y}^{-1}\mathbf{f}_{B_y}^H\right\|_\infty^{-1}\right\}\\
 &=2\min\left\{\frac{2}{3}, \frac{11}{24} \right\}=\frac{11}{12}.
\end{align} 

In the first equality above we are using the fact that the function $f_S=\mathbf{f}_{B_x}\mathbf{f}_{A_x}^{-1}\mathbf{f}_{B_x}^H+\mathbf{f}_{B_y}\mathbf{f}_{A_y}^{-1}\mathbf{f}_{B_y}^H$ is a scalar valued function whose maximum is $\frac{1}{4}$ corresponding to the point $(0,\pi)$ (equiv. $(\pi,0)$).
Figure \ref{fig:plot_f_S} shows the plot of ${f}_{S}$ on a uniform grid over $[0,\pi]^2$.

\begin{figure}[htb]
\centering
\includegraphics[width=0.50\textwidth]{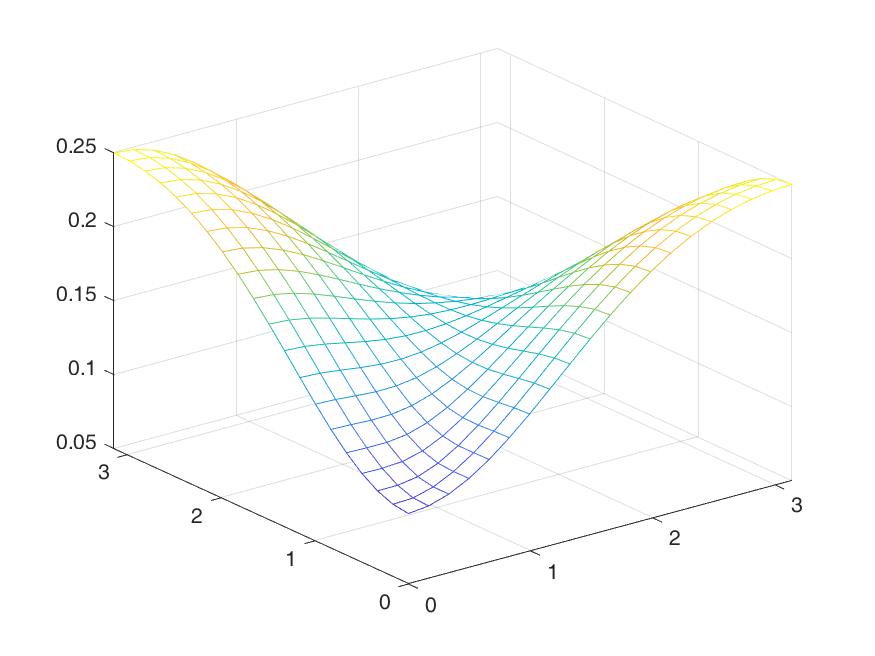}
\caption{Plot of the  function ${f}_{S}$ on a uniform grid over $[0,\pi]^2$ with $n=20$ points.}\label{fig:plot_f_S}
\end{figure}

Therefore, Theorem \ref{thm:Notay_symbol_2D_block} ensures that the TGM applied to the system having $\mathcal{\hat{A}}$ as coefficient matrix converges with the choice $\omega\in (0, 11/12) $.

\section{Extension to MGM}\label{sec:procedure}
In Section \ref{sec:saddle_multi_block} we presented a TGM procedure for solving a system of the form (\ref{eq:saddle_notay}) consisting in the transformation (\ref{eq:def_LAU_2D_block}) and the application of the projecting strategy given by $\mathcal{P}$. In this section we propose a multigrid strategy with more than two grids.
Moreover, in order to analyse the multigrid method for multiple grids, we need to study the problem at the coarser levels.
\subsection{A 3-step procedure}
Firstly, we define the CGC introducing a procedure consisting of 3 steps (3SP):
\begin{itemize}
\item[(i)] transform the system into a new one by making use of the two invertible matrices $\mathcal{L}$ and $\mathcal{U}$.
\item[(ii)] Apply the projecting strategy given by $\mathcal{P}$.
\item[(iii)] “Symmetrize" the problem with a proper variable transformation.
\end{itemize}
First of all, note that after step one the transformed system is not symmetric. However, the block $\hat{C}=B (2\alpha D_{\tilde{A}}^{-1}-\alpha^2D_{\tilde{A}}^{-1}\tilde{A}D_{\tilde{A}}^{-1})B^T$ is a SPD matrix as long as the parameter $\alpha$ is such that $\alpha< 2(\lambda_{\rm max}(D_{\tilde{A}}^{-1}A))^{-1}$. Equivalently, the matrix in block position $(1,1)$ is SPD, since it is given by $\tilde{A}$.

After the application of the grid transfer operator 
\begin{equation*}
\mathcal{P}=\begin{bmatrix}
P_{\tilde{A}}& \\
& P_{\hat{C}}
\end{bmatrix},
\end{equation*}
the matrix at the coarse level has the form
\begin{equation}\label{eq:projectsystem}
\begin{bmatrix}
P_{\tilde{A}}^T\tilde{A}P_{\tilde{A}} &P_{\tilde{A}}^T (I_{(d-1)n^2}-\alpha\tilde{A}D_{\tilde{A}}^{-1})B^TP_{\hat{C}}\\
-P_{\hat{C}}^TB(I_{(d-1)n^2}-\alpha D_{\tilde{A}}^{-1}\tilde{A})P_{\tilde{A}} & P_{\hat{C}}^TCP_{\hat{C}}+P_{\hat{C}}^TB (2\alpha D_{\tilde{A}}^{-1}-\alpha^2D_{\tilde{A}}^{-1}\tilde{A}D_{\tilde{A}}^{-1})B^TP_{\hat{C}}
\end{bmatrix},
\end{equation}
which is a matrix of a form which resembles that of the system (\ref{eq:saddle_notay}), apart from the signs of the blocks in position $(2,1)$ and $(2,2)$. 
This explains why the third step of our procedure consists in the variable transformation obtained multiplying both (\ref{eq:projectsystem}) and the right hand side on the left by the matrix
\begin{equation}
\begin{bmatrix}
I_{2s_an^2}& \\
 & -I_{s_cn^2}
\end{bmatrix}.
\end{equation}
This last step restores the symmetry of the matrix and gives to (\ref{eq:projectsystem}) the same block form of the original coefficient matrix so we can proceed recursively, defining a V-cycle procedure.

Precisely,  the base of the recursion is the original system written as 
\begin{equation}\label{eq:Astep0}
\mathcal{A}_{N_0}\{0\} = \left[\begin{array}{ccc}
	\tilde{A}\{0\} & B\{0\}^T \\
	B\{0\} & -C\{0\}
	\end{array}\right],
\end{equation}
with  $\tilde{A}\{0\}\in \mathbb{R}^{s_an_0^2\times s_an_0^2}$ SPD, $C\{0\} \in \mathbb{R}^{s_cn_0^2\times s_cn_0^2}$ non-negative definite, and $B\{0\}\in \mathbb{R}^{s_cn_0^2\times 2s_an_0^2}$.

At level $\ell$, we have $P_{\tilde{A}\{\ell\}}\in\mathbb{R}^{2s_an_{\ell}^2 \times2s_an_{\ell+1}^2}$ and $P_{\hat{C}\{\ell\}}\in\mathbb{R}^{s_cn_{\ell}^2\times s_cn_{\ell+1}^2}$  the prolongation operators chosen for solving efficiently the scalar systems with coefficient matrix   $\tilde{A}\{\ell\}$ and $\hat{C}\{\ell\}=C\{\ell\}+B\{\ell\} (2\alpha_\ell D_{\tilde{A}\{\ell\}}^{-1}-\alpha_\ell^2D_{\tilde{A}\{\ell\}}^{-1}\tilde{A}\{\ell\}D_{\tilde{A}\{\ell\}}^{-1})B\{\ell\}^T$, respectively, where $\alpha_\ell$ is such that $\alpha_\ell<2\lambda_{\max}\left(D_{\tilde{A}\{\ell\}}^{-1}\tilde{A}\right)^{-1}$.
Then, the inductive step of our recursive definition is given by
\begin{equation}
\mathcal{A}_{N_{\ell+1}}\{\ell+1\} = \left[\begin{array}{ccc}
	\tilde{A}\{\ell+1\} & B\{\ell+1\}^T \\
	B\{\ell+1\} & -\hat{C}\{\ell+1\}
	\end{array}\right].
\end{equation}
where
\begin{small}
\begin{equation}\label{eq:A_recursive_def}
	\tilde{A}\{\ell+1\}=P_{\tilde{A}\{\ell\}}^T\tilde{A}\{\ell\}P_{\tilde{A}\{\ell\}},
\end{equation}
\begin{equation}\label{eq:B_recursive_def}
	B\{\ell+1\}=P_{\hat{C}\{\ell\}}^TB\{\ell\}(I_{2s_an_\ell^2}-\alpha_\ell D_{\tilde{A}\{\ell\}}^{-1}\tilde{A}\{\ell\})P_{\tilde{A}\{\ell\}}, 
\end{equation}
\begin{equation}\label{eq:C_recursive_def}
	\hat{C}\{\ell+1\}=\left(P_{\hat{C}\{\ell\}}^T\hat{C}\{\ell\}P_{\hat{C}\{\ell\}}+P_{\hat{C}}^TB\{\ell\} (2\alpha_\ell D_{\tilde{A}\{\ell\}}^{-1}-\alpha_\ell^2D_{\tilde{A}\{\ell\}}^{-1}\tilde{A}\{\ell\}D_{\tilde{A}\{\ell\}}^{-1})B\{\ell\}^TP_{\hat{C}\{\ell\}}\right).
\end{equation}
\end{small}

For the considerations that we made when we described our 3SP, for each $\ell$ we have that $\tilde{A}\{\ell+1\}\in \mathbb{R}^{2s_an_{\ell+1}^2\times 2s_an_{\ell+1}^2}$ is SPD, $C\{{\ell+1}\} \in \mathbb{R}^{s_cn_{\ell+1}^2\times s_cn_{\ell+1}^2}$ is non-negative definite, and $B\{{\ell+1}\}\in \mathbb{R}^{s_cn_{\ell+1}^2\times 2s_an_\ell^2}$.

This procedure can be seen globally as one block multigrid algorithm  with restriction and prolongation operators $R\{\ell\}$ and $P\{\ell\}$.

\begin{Proposition}\label{prop:proj}
Consider the three steps procedure described by items (i)-(ii). At level $\ell$, define the prolongation and restriction operators as
\begin{multline}
R\{\ell\}=\begin{bmatrix}
I_{2s_an^2\{\ell\}}& \\
\alpha\{\ell\}P_{\hat{C}\{\ell\}}^TB\{\ell-1\}  D_{\tilde{A}\{\ell\}}^{-1}  P_{\tilde{A}\{\ell\}}&-I_{s_cn^2\{\ell\}}\end{bmatrix}\cdot \\
 \cdot\begin{bmatrix}
I_{2s_an^2\{\ell\}}& \\
& -I_{s_cn^2\{\ell\}}& \\
\end{bmatrix}
\begin{bmatrix}
P^T_{\tilde{A}\{\ell\}}&\\
& P^T_{\tilde{A}\{\ell\}}
\end{bmatrix}
\end{multline}
\begin{equation}
P\{\ell\}=\begin{bmatrix}
P^T_{\tilde{A}\{\ell\}}&\\
& P^T_{\hat{C}\{\ell\}}
\end{bmatrix}\begin{bmatrix}
I_{2s_an^2\{\ell\}}& -\alpha\{\ell\}P_{\tilde{A}\{\ell\}}^T D_{\tilde{A}\{\ell\}}^{-1} B^T\{\ell-1\}P_{\tilde{A}\{\ell\}} \\
& I_{s_cn^2\{\ell\}}
\end{bmatrix}
\end{equation}
After $\ell+1$ iterations of 3SP, the resulting coefficient matrix $\mathcal{A}_{N_{\ell+1}}\{\ell+1\}$ is given by
\begin{equation}
\mathcal{A}_{N_{\ell+1}}\{\ell+1\}=R\{\ell\}\begin{bmatrix}
\tilde{A}\{\ell\} & B^T\{\ell\} \\
-B\{\ell\} & \hat{C}\{\ell\} 
\end{bmatrix}P\{\ell\}.
\end{equation}

%
%
\end{Proposition}
\begin{proof}
The assertion follows immediately by direct computation.
%
%
%
%
\end{proof}

\subsection{Smoother and grid transfer operator at coarser levels}\label{ssec:smoth_proj_mgm}
In order to develop a multigrid method as described in the previous subsection, we need to choose a sequence of smoothing parameters $\omega_\ell$ and to construct two sequences of grid transfer operators whose elements are $P_{\tilde{A}\{\ell\}}$ and $P_{\hat{C}\{\ell\}}$, associated respectively with ${\tilde{A}\{\ell\}}$ and ${\hat{C}\{\ell\}}$.

In the circulant setting, the matrices ${\tilde{A}\{\ell\}}$ and ${\hat{C}\{\ell\}}$ are still circulant matrices and we denote by $\mathbf{f}_{\tilde{A}\{\ell\}}$ and $\mathbf{f}_{\hat{C}\{\ell\}}$ the respective generating functions. In the Toeplitz setting, the Toeplitz structure is lost at coarser levels, in particular for the block ${\hat{C}\{\ell\}}$. However, ${\hat{C}\{\ell\}}$ can still be associated with the symbol $\mathbf{f}_{\hat{C}\{\ell\}}$, computed like in the circulant case. This is possible because the corrections at the corners have a rank which remains bounded as $\ell$ increases, as a consequence of the bound on the matrix bandwidth on which we will comment in the next section. 

In principle, our strategy at level $\ell$ is to choose the smoothing parameter $\omega_\ell$ and the grid transfer operators $P_{\tilde{A}\{\ell\}}$ and $P_{\hat{C}\{\ell\}}$ such that the TGM applied to the system with coefficient matrix $\mathcal{A}_{N_{\ell}}\{\ell\}$ converges. The convergence of the ``coarser TGMs'' is a necessary condition for the convergence of full multigrid methods such as the W-cycle and the V-cycle, see \cite{ADS} for a discussion on the topic.

Combining all the previous considerations, for constructing an MGM for $\mathcal{A}$ in \eqref{sec:saddle_multi_block} we aim at fulfilling the assumptions of Theorem \ref{thm:Notay_symbol_2D_block} applied to the matrix $\mathcal{A}_{N_{\ell}}\{\ell\}$ for all $\ell$. A rigorous theoretical analysis of the properties of the symbols $\mathbf{f}_{\tilde{A}\{\ell\}}$ and $\mathbf{f}_{\hat{C}\{\ell\}}$ in the block multilevel setting would require a significant amount of computations, undermining the readability of the paper, so we omit it. Concerning $\mathbf{f}_{\tilde{A}\{\ell\}}$, the reader can refer to \cite{block_multigrid2}. In the applications, it is usually possible to choose at the finest level the trigonometric polynomials $\mathbf{p}_{A_x}$, $\mathbf{p}_{A_y}$, and $\mathbf{p}_{\hat{C}}$ in such a way that the same choice is suitable for all levels, that is, the properties of $\mathbf{f}_{\tilde{A}\{\ell\}}$ and $\mathbf{f}_{\hat{C}\{\ell\}}$ which are relevant for the TGM convergence remain unchanged level after level. For instance, this is the case for the problem and TGM analysed in Section \ref{sec:stokes}, as we will remark in Section \ref{sec:numerical}.

Regarding the choice of $\omega_\ell$, considering equation \eqref{eq:choice_of_omega} we derive the following condition:
\begin{multline*}\label{eq:choice_of_omega_ell}
 \omega_\ell<2\min \left\{2\alpha_\ell -\alpha_\ell^2\max_j\left\{ \frac{\|\mathbf{f}_{A_x\{\ell\}}\|_\infty}{{\hat{a}_0^{(j,j)}(\mathbf{f}_{A_x\{\ell\}})}}, \frac{\|\mathbf{f}_{A_y\{\ell\}}\|_\infty}{{\hat{a}_0^{(j,j)}(\mathbf{f}_{A_y\{\ell\}})}}\right\},\right.\\
 \left.\max_j\hat{a}_0^{(j,j)}(\mathbf{f}_{\hat{C}\{\ell\}})\left\|\mathbf{f}_{C\{\ell\}}+\mathbf{f}_{B_x\{\ell\}}\mathbf{f}_{A_x\{\ell\}}^{-1}\mathbf{f}_{B_x\{\ell\}}^H+\mathbf{f}_{B_y\{\ell\}}\mathbf{f}_{A_y\{\ell\}}^{-1}\mathbf{f}_{B_y\{\ell\}}^H\right\|_\infty^{-1}\right\},
\end{multline*} 
which suggests how the smoothing parameter should be chosen at each level for obtaining a convergent W-cycle.

\subsection{Bandwidth at coarser levels}

When the generating function is a (multivariate matrix-valued) trigonometric polynomial, the matrix-vector product with the associated circulant or Toeplitz matrix has a computational cost linear in the matrix size. If this is the case for the blocks $A, B$, and $C$, the matrix-vector product with matrix $\mathcal{A}$ in \eqref{eq:problem_matrix} is still linear in the matrix size. Therefore, in order to bound the computational cost of an iteration of the multigrid method, it is fundamental to make sure that the degree of the generating trigonometric polynomials is bounded as the level increases.
 {In the following  report the result in the unilevel scalar case which was proved in \cite{AD} and \cite{saddle_DBFF}. }


\begin{Lemma}\cite[Proposition 2]{AD} \label{lem:coarse_degree_A}
Let ${A\{\ell\}}$ be defined in \eqref{eq:A_recursive_def}, with $P_{A{\{\ell\}}}=\mathcal{C}_{n_\ell}(p_{A})K_{n_\ell}^T$, where $f_{A\{0\}}$ and $p_A$ are trigonometric polynomials of degree $z_{A\{0\}}$ and $q_A$, respectively. 
Then $f_{A\{\ell\}}$ is a trigonometric polynomial of degree $z_{A\{\ell\}}$ such that:
\begin{enumerate}
	\item $z_{A\{\ell\}} \le \max(z_{A\{0\}}, 2q_A)$ for all $\ell$;
	\item $z_{A\{\ell\}} \le 2q_A$ for $\ell$ large enough.
\end{enumerate}
\end{Lemma}
\begin{Lemma}\cite{saddle_DBFF} \label{lem:coarse_degree_B}
Consider the matrices ${B}\{\ell\}=\mathcal{C}_{n_\ell}(f_{{B}\{\ell\}})$ and ${C}\{\ell\}=\mathcal{C}_{n_\ell}(f_{{C}\{\ell\}})$ defined by formulas (\ref{eq:B_recursive_def})-\eqref{eq:C_recursive_def}, with $P_{A{\{\ell\}}}=\mathcal{C}_{n_\ell}(p_{A})K_{n_\ell}^T$ and $P_{\hat{C}{\{\ell\}}}=\mathcal{C}_{n_\ell}(p_{\hat{C}})K_{n_\ell}^T$. Let $z_{B\{\ell\}}$ and $z_{C\{\ell\}}$ be the degrees of $f_{{B}\{\ell\}}$ and $f_{{C}\{\ell\}}$, respectively. Let $q=\max\{q_A, q_{\hat{C}}\},$ where $q_A$ and $q_{\hat{C}}$ are the polynomial degrees associated with $p_{A}$ and $p_{\hat{C}}$, respectively. 

Then, for $\ell$ large enough, it holds 
\begin{enumerate}
\item $z_{B\{\ell\}}\le \max (2z_{B\{0\}}, 4q ),$
\item $z_{C\{\ell\}}\le \max(4 z_{B\{0\}},6q,2z_{C\{0\}}).$
\end{enumerate}
\end{Lemma} 

Seeing a matrix-valued trigonometric polynomial as a matrix whose elements are trigonometric polynomials, it is straightforward to extend the latter results to block circulant matrices. Indeed, an element of a generating matrix-valued trigonometric polynomial $\mathbf{f}_{{A}\{\ell\}}$, $\mathbf{f}_{{B}\{\ell\}}$ or $\mathbf{f}_{{C}\{\ell\}}$ at coarser levels is the sum of trigonometric polynomials whose degree is bounded by Lemmas \ref{lem:coarse_degree_A}--\ref{lem:coarse_degree_B}.


\section{Numerical Examples}\label{sec:numerical}

In this section we present the numerical confirmation of the efficiency of the proposed block multigrid method when applied on the linear system (\ref{eq:syst_q1_iso_q2}).

 We exploit the theoretical results presented in Section \ref{sec:saddle_multi_block} and Subsection \ref{ssec:A} to choose the smoothing parameter $\omega$ and the grid transfer  operator $\mathcal{P}=\begin{bmatrix}
P_{\tilde{A}}& \\
& P_{\hat{C}}
\end{bmatrix}$. 
Precisely, $P_{\tilde{A}}$ and $P_{C}$ are defined by formulae (\ref{eq:P_Atilde}) and (\ref{eq:P_C}), respectively. The choice of the Jacobi smoothing parameter $\omega$, which completes the multigrid procedure, is related with the range of admissibility in (\ref{eq:choice_omega_num}), that is $\omega\in (0, 11/12) $.

 All the tests are performed using MATLAB 2021a and the error equation at the coarsest level is solved with the MATLAB backslash function.
In all the experiments we use the standard stopping criterion  $\frac{\|r^{(k)}\|_2}{\|\mathbf{b}\|_2}<\epsilon$, where $r^{(k)}=\mathbf{b}-{\mathcal{M}}\mathbf{x}^{(k)}$ and $\epsilon=10^{-6}$. Concerning $\mathbf{b}$, we consider the case where the true solution $\mathbf{x}$ of the linear system ${\mathcal{M}}\mathbf{x} = \mathbf{b}$ is a uniform sampling of $\sin(4t)+\cos(6t)+1$ on $[0,\pi]$ and we compute the right-hand side $\mathbf{b}$ as $\mathbf{b} = {\mathcal{M}}\mathbf{x}$.

Since the grid-transfer operator $P_{{A}}$ has a geometrical meaning---that is, it represents standard bilinear interpolation---we take the partial dimension $n$ of the form $2^{t}+1$, $t=5,\dots,8$.

We first consider a TGM consisting of only 1 post smoothing step of damped Jacobi. We test the efficiency of the TGM with three different values, $\omega= 2/5, 3/5, 4/5$. The iterations shown in Table \ref{tab:TGM_Jacs} remain almost constant as the matrix size increases, confirming the theoretical optimal convergence rate.

\begin{table}
		\begin{center}
		\caption{ Two-Grid iterations for the matrix ${\mathcal{A}}_N$ with $\epsilon= 10^{-6}$ and $\omega_{post}=[2/5, 3/5, 4/5]$.}
		\begin{tabular}{c|c|c|c}		
		$N=9\cdot (2^t+1)^2 $ & \multicolumn{3}{c}{$\#$ Iterations}\\
			\hline
			{$t$} & {$\omega_{post}=2/5$} & {$\omega_{post}=3/5$} & {$\omega_{post}=4/5$} \\
			\hline
			5& 38 & 23& 17\\
			6& 36 & 22& 16\\
			7& 35 & 20& 15 \\
			8& 35 & 19& 15 \\	
	
	\end{tabular}
	\label{tab:TGM_Jacs}
		\end{center}
\end{table}

As a second experiment we consider different number of steps of pre/post smoother. In order to damp the error both in the middle and in the high frequencies, we take a different parameter for the pre-smoother and the post-smoother (when both are present). In particular, we choose $\omega_{ \rm post} = 4/5$ and $\omega_{ \rm pre} = 3/5$, since they give the fastest convergence over the considered sample of admissible values. The convergence rate in all these cases is optimal (see Table \ref{tab:TGM_Jacs_different_steps}) since the presence of the pre smoothing or of additional smoothing steps can accelerate the convergence of the TGM.

\begin{table}
		\begin{center}
		\caption{ Two-Grid iterations for the matrix ${\mathcal{A}}_N$ with $\epsilon= 10^{-6}$ and with different number of pre/post smoothing steps. If both present,  $\omega_{ \rm post} = 4/5$ and $\omega_{ \rm pre} = 3/5$.}
		\begin{tabular}{c|c|c|c|c}		
		$N=9\cdot (2^t+1)^2 $ & \multicolumn{4}{c}{$\#$ Iterations}\\
			\hline
			{$t$} & TGM$(0,1)$ &  TGM$(1,0)$ &  TGM$(1,1)$ &  TGM$(2,2)$ \\
			\hline
			5& 17 & 20 & 16 &15 \\
			6& 16 & 19 & 16 &15\\
			7& 15 & 18 & 15& 14\\
			8& 15 & 17  & 14 & 13\\	
	
	\end{tabular}
	\label{tab:TGM_Jacs_different_steps}
		\end{center}
\end{table}

In Section \ref{sec:procedure} we described the structure of the blocks at coarser levels and how the multigrid ingredients should be chosen when considering more than two grids. 
 {In particular, we consider a number of levels depending by the system matrix size according to the following rule: if the size is of the form $s(2^t+1)^2$, the number of levels after which we stop the iteration is maximum between 1 and $t-3$. 
When considering more levels,}
 our choice of trigonometric polynomials $\mathbf{p}_{A_x}$, $\mathbf{p}_{A_y}$, and $\mathbf{p}_{\hat{C}}$ ensures that the spectral properties of $\tilde{A}\{\ell\}$ and $\hat{C}\{\ell\}$ remain unchanged at coarser levels, as we verified numerically using the associated symbols. Regarding the choice of $\omega_\ell$, we numerically verified that choosing at all levels the smoothing parameter which is suitable for the first level is a computationally cheaper and more robust choice in order to accelerate the convergence of the V-cycle. This approach was already taken in \cite{saddle_DBFF}, where a comparison with an adaptive choice of $\omega_\ell$ was provided.  In Table \ref{tab:TGM_W_V_Jacs} we numerically prove the independence of the convergence rate with respect to the matrix size also for the W-cycle and V-cycle methods using 2 steps of pre/post smoother, with the choice  $\omega_{ \rm post} = 4/5$ and $\omega_{ \rm pre} = 3/5$ at all levels. {In addition, we report the computational times for solving the linear system and the  setup times for the computation of the grid transfer operators. Although W-cycle and V-cycle require the computation of the grid transfer operators for more levels, their setup times remain relatively low. Then the two methods, are preferable in practical cases with respect to the TGM. In particular, V-cycle method  shows  best computational times at the price of just few iterations more. }

\begin{table}
		\begin{center}
		\caption{ Comparison of TGM, W-cycle and V-cycle number of iterations (IT), { CPU times $T(s)$ and setup times $T_{\rm set}(s)$} for the matrix ${\mathcal{A}}_N$, $N=9\cdot (2^t+1)^2 $  with 2 iterations of pre/post smoother of Jacobi and the choice  $\omega_{ \rm post} = 4/5$ and $\omega_{ \rm pre} = 3/5$, $\epsilon= 10^{-6}$. }
		\begin{tabular}{c|c|c|c|c|c|c|c|c|c}		
		&  \multicolumn{3}{c}{ TGM} &\multicolumn{3}{c}{ W-cycle} & \multicolumn{3}{c}{ V-cycle}\\
			\hline
			{$t$} & {IT} &  { $T(s)$} & $T_{\rm set}(s)$& { IT} &  { $T(s)$} & $T_{\rm set}(s)$& {IT} &  { $T(s)$} & $T_{\rm set}(s)$ \\
			\hline
			5& 15& 1.058 &0.046  & 15&   2.100 &0.046  & 15&1.098&0.046  \\
			6& 15& 5.655 &0.183 & 15&5.431 &0.192& 15&1.722&0.192 \\
			7& 14& 31.314 &0.639 & 14& 11.790 &0.854 & 15&3.083&0.854 \\
			8& 13&160.615 &2.450 & 13&30.660&3.566& 16&9.153 &3.566 \\	
	
	\end{tabular}
	\label{tab:TGM_W_V_Jacs}
		\end{center}
\end{table}


In many works \cite{MR3564863,MR3543002, MR2220678, MR3689933, MR2155549}
problems with such saddle-point structure have been treated, exploiting particular  preconditioned iterative strategies. In particular,  a common choice is that of constructing preconditioning for two-by-two block systems involving the Schur complement, whose computation requires the solution of the linear system with coefficient matrix the positive definite part ${\tilde{A}}$ of ${\mathcal{A}}$. Then, the theoretical findings of Subsection \ref{ssec:A} on $\tilde{A}$ and on the grid transfer operator $P_{\tilde{A}\{\ell\}}$ can be exploited also for strategies different from the proposed ones. Moreover, Subsection \ref{ssec:A} suggests valid tools that can be used for constructing grid transfer operators, since it shows how to proceed for verifying the conditions which lead to convergence and optimality. Furthermore, another approach to solve saddle-point problem of the form (\ref{eq:saddle_notay}) consists of multigrid  methods in which the smoother takes into account a special coupling, represented by the off-diagonal blocks in \eqref{eq:saddle_notay}. Then, as a last test we compare our results with the standard  Vanka smoother \cite{MR848451}. The analysis of block smoothers like the one mentioned can be found in \cite{MR3795547,MR2840198,MR3488076}.
In Table \ref{tab:V_Vanka_Jac} we compare the iterations of the V-cycle method using the two different smoothers: on the left we show the iterations for the system $\mathcal{A}$ considering 2 pre/post smooting steps of  Vanka method. {On the right we report the iterations related to the V-cycle method applied to system $\tilde{\mathcal{A}}=\mathcal{LAU}$, whose efficiency has been already discussed. Even if the Vanka smoother performs well in combination with the linear interpolation projecting strategy, we stress that the computational cost of a single Vanka-based MGM iteration is higher than the one of an iteration of our procedure, as shown by the computational times in Table \ref{tab:V_Vanka_Jac}.}

\begin{table}
		\begin{center}
		\caption{
    V-cycle iterations (IT), { CPU times $T(s)$ and setup times $T_{\rm set}(s)$} for the matrix ${{A}}_N$,  with 2 iterations of pre/post smoother of Vanka method (left) and for the matrix ${\mathcal{A}}_N$,   with 2 iterations of pre/post smoother Jacobi (right) and the choice  $\omega_{ \rm post} = 4/5$ and $\omega_{ \rm pre} = 3/5$. $\epsilon= 10^{-6}$ and $N=9\cdot (2^t+1)^2 $.}

		\begin{tabular}{c|c|c|c|c|c|c}		
	 	&  \multicolumn{3}{c}{ V-cycle with Vanka for ${A}_N$} &\multicolumn{3}{c}{V-cycle with Jacobi for ${\mathcal{A}}_N$ } \\
			\hline
			{$t$} & {IT} &  { $T(s)$} & $T_{\rm set}(s)$ & { IT} &  { $T(s)$} & $T_{\rm set}(s)$  \\
			\hline
			5& 10& 3.285	&0.018 & 15	&       1.098	&0.046\\
			6& 12& 75.400 	&0.089 &15		&    1.722		&0.192\\
			7& 12& 867.798	&0.168 &15		&   3.083		&0.854\\
			8& 11&11063.057 & 0.724&	16	&     9.153 	&3.566\\	
	
	\end{tabular}

	\label{tab:V_Vanka_Jac}
		\end{center}
\end{table}

 {
Furthermore, we test our block multigrid method as a preconditioner for solving the linear system with the \textit{GMRES} method. Among the existing Krylov subspace methods, we choose the GMRES because the transformed global system matrix $\tilde{\mathcal{A}}=\mathcal{LAU}$ is not symmetric. Our resulting method is a Preconditioned GMRES, denoted by $P_{\tiny{\mbox{V-cycle}}}$--GMRES,  where the preconditioner is one iteration of the V-cycle method whose behaviour as a standalone solver was shown in Table \ref{tab:TGM_W_V_Jacs}. In all the experiments we use the built-in Matlab function \textit{gmres}.
In the second column of Table \ref{tab:gmres_tepGlobal} we show the number of iterations needed for the convergence of the $P_{\tiny{\mbox{V-cycle}}}$--GMRES when increasing the size of $\tilde{\mathcal{A}}$.
We use the zero vector as initial guess and we set the tolerance of the method to $\epsilon=10^{-6}$.
Precisely the third column of Table \ref{tab:gmres_tepGlobal} shows that the number of iterations of the $P_{\tiny{\mbox{V-cycle}}}$--GMRES remains bounded  when increasing the matrix size of the system, i.e the number of levels $t$.
}

 {
In Table \ref{tab:gmres_tepGlobal} we report also the comparison with a state-of-the-art preconditioner for the Stokes problem, see \cite{MR2155549}.
We consider a block diagonal matrix with $\tilde{A}$ as $(1,1)$-block and having in block $(2,2)$ the pressure mass matrix $M$, which in our FEM setting is the properly sized bi-level scalar Toeplitz matrix generatred by
\[
    {f}_{M}(\theta_1,\theta_2)=\frac{1}{18}\left(8+4\cos(\theta_1)+4\cos(\theta_2)+\cos(\theta_1+\theta_2)+\cos(\theta_1-\theta_2)\right).
\]
We apply the approximate inverse of such preconditioner for the linear system with matrix $\mathcal{A}$ performing 5 V-cycle iterations for the block $\tilde{A}$ and 5 V-cycle iterations for the block $M$. The generating function of $M$ is strictly positive, hence it is straightforward to find a suitable grid transfer operator according to \cite{Serra_Possio}. Concerning $\tilde{A}$, we remark that the findings in \ref{ssec:A} suggest the Jacobi smoothing parameter and grid transfer operator that can be chosen to achieve convergence with an optimal convergence rate, according to \cite{block_multigrid2}. In particular the grid transfer operator can be the one studied in Subsection \ref{ssec:A}. Finally,  we highlight that while  the exact preconditioner is symmetric positive definite, our multigrid procedures are not, hence, we use again the GMRES.
}

\begin{table}\label{tab:gmres_tepGlobal}
		\begin{center}
		\caption{ {$P_{\tiny{\mbox{V-cycle}}}$--GMRES and $P_{\diag}$--GMRES   number of iterations (IT),  CPU times $T(s)$ and setup times $T_{\rm set}(s)$ for the matrix ${\mathcal{A}}_N$, $N=9\cdot (2^t+1)^2 $ for the global matrix ${\mathcal{A}}_N$.}}
		\begin{tabular}{c|c|c|c|c|c|c}		
		&  \multicolumn{3}{c}{ {$P_{\tiny{\mbox{V-cycle}}}$--GMRES}} &\multicolumn{3}{c}{ {$P_{\diag}$--GMRES}} \\
			\hline
			{$t$} & {IT} &  { $T(s)$} & $T_{\rm set}(s)$& { IT} &  { $T(s)$} & $T_{\rm set}(s)$ \\
			\hline
			5& 12&1.056  &0.046 & 58 & 1.468 & 0.086 \\
			6& 13&1.835  &0.192 & 63& 3.016 & 0.249 \\
			7& 14&3.466  &0.854 & 69& 14.459 & 0.632  \\
			8& 17&11.707 &3.566 & 74& 55.492 & 1.858 \\	

	\end{tabular}
		\end{center}
\end{table}

\section{Conclusions}
In this paper we studied a symbol based convergence analysis for problems that have a saddle-point form with block Toeplitz structure. From theoretical point of view, we extended the sufficient conditions for the TGM convergence and the choice of smoothing parameters to the case where the associated generating function $\mathbf{f}$ is a matrix-valued function. Moreover, we demonstrated  that a recursive procedure is possible, preserving the saddle-point structure at the coarser levels, after proper symmetrization.
In order to test the efficiency of the theory we considered the saddle-point system involved in the numerical solution of the Stokes problem in 2D discretized with FEM.

 In particular, we exploited the properties of the  generating functions of the blocks and  how these can be used  to develop a block multigrid method which is convergent and optimal. Numerically, the resulting methods show the expected convergence behavior, confirming the validity of our analysis. 
 
 The promising results suggest further investigation into the efficiency of the procedure for other practical applications with challenging problems. The analysis is noteworthy as it provides with reasonable effort a theoretical guarantee of convergence by solely examining the symbols of the matrix sequences involved.

\section*{Acknowledgments}
The work of the second, third, and fourth authors is partly supported by “Gruppo Nazionale per il Calcolo Scientifico" (GNCS-INdAM).
{Moreover, the work of Isabella Furci was carried out within the framework of the project “A multiscale integrated approach to the study of the nervous system in health and disease (MNESYS)” and has been supported by European Union - NextGenerationEU.}

\bibliographystyle{abbrv}
\bibliography{biblio}
\end{document}